\documentclass[12pt]{amsart}
\usepackage{amsmath,amsthm,amsfonts,amscd,amssymb,eucal,latexsym,mathrsfs, appendix, yhmath, fancyhdr, setspace}
\usepackage[numbers,sort&compress]{natbib}
\usepackage[all,cmtip]{xy}

\newtheorem{theorem}{Theorem}[section]
\newtheorem{corollary}[theorem]{Corollary}
\newtheorem{lemma}[theorem]{Lemma}
\newtheorem{proposition}[theorem]{Proposition}

\theoremstyle{definition}
\newtheorem{definition}[theorem]{Definition}
\newtheorem{remark}[theorem]{Remark}

\newtheorem{example}[theorem]{Example}
\newtheorem{question}[theorem]{Question}

\newtheorem{conjecture}[theorem]{Conjecture}

\newcommand{\rank}{{\rm rank}}

\newcommand{\im}{{\rm im}}

\newcommand{\mdim}{{\rm mdim}}
\newcommand{\ord}{{\rm ord}}

\newcommand{\cB}{{\mathcal B}}
\newcommand{\cH}{{\mathcal H}}

\newcommand{\cW}{{\mathcal W}}
\newcommand{\cD}{{\mathcal D}}

\newcommand{\cC}{{\mathcal C}}
\newcommand{\cU}{{\mathcal U}}
\newcommand{\cM}{{\mathcal M}}
\newcommand{\cN}{{\mathcal N}}

\newcommand{\cF}{{\mathcal F}}
\newcommand{\cV}{{\mathcal V}}

\newcommand{\fM}{{\mathfrak M}}
\newcommand{\fN}{{\mathfrak N}}
\newcommand{\Cb}{{\mathbb C}}

\newcommand{\Zb}{{\mathbb Z}}

\newcommand{\Qb}{{\mathbb Q}}
\newcommand{\Rb}{{\mathbb R}}
\newcommand{\Nb}{{\mathbb N}}

\newcommand{\sM}{{\mathscr M}}
\newcommand{\sN}{{\mathscr N}}

\newcommand{\diam}{{\rm diam}}

\newcommand{\tr}{{\rm tr}}

\newcommand{\GL}{{\rm GL}}

\newcommand{\rh}{{\rm h}}

\newcommand{\rM}{{\rm M}}

\newcommand{\mrank}{{\rm mrank}}

\newcommand{\vrank}{{\rm vrank}}
\newcommand{\erank}{{\rm erank}}

\allowdisplaybreaks

\begin{document}

\title{Mean Dimension, Mean Rank, and von Neumann-L\"{u}ck Rank}

\author{Hanfeng Li}
\author{Bingbing Liang}

\address{\hskip-\parindent
H.L., Center of Mathematics, Chongqing University,
Chongqing 401331, China.
\break
Department of Mathematics, SUNY at Buffalo,
Buffalo, NY 14260-2900, U.S.A.}
\email{hfli@math.buffalo.edu}

\address{\hskip-\parindent
B.L., Department of Mathematics, SUNY at Buffalo,
Buffalo NY 14260-2900, U.S.A.}
\email{bliang@buffalo.edu}


\date{April 20, 2015}

\begin{abstract}
We introduce an invariant, called mean rank, for any module $\cM$ of the integral group ring of a discrete amenable group $\Gamma$, as an analogue of the rank of an abelian group. It is shown that the mean dimension of the induced $\Gamma$-action on the Pontryagin dual of $\cM$, the mean rank of $\cM$, and the von Neumann-L\"{u}ck rank of $\cM$ all coincide.

As applications, we establish an addition formula for mean dimension of algebraic actions, prove the analogue of the Pontryagin-Schnirelmnn theorem for algebraic actions, and show that for elementary amenable groups with an upper bound on the orders of finite subgroups, algebraic actions with zero mean dimension are
inverse limits of finite entropy actions.
\end{abstract}

\maketitle

\tableofcontents

\section{Introduction} \label{S-introduction}

Mean dimension was introduced by Gromov \cite{Gromov}, and developed systematically by Lindenstrauss and Weiss \cite{LW}, as an invariant for continuous actions of countable amenable groups on compact metrizable spaces. It is a dynamical analogue of the covering dimension, and closely related to the topological entropy. Lindenstrauss and Weiss used it to show that certain minimal homeomorphism does not embed into the shift action with symbol the unit interval \cite{LW}. It has received much attention in the last several years \cite{CK, Gutman, Krieger06, Krieger09, Li13, Lindenstrauss, Tsukamoto08, Tsukamoto09}.

The von Neumann-L\"{u}ck dimension was originally defined for finitely generated projective modules over the group von Neumann algebra $\cN\Gamma$ of a discrete group $\Gamma$, and later extended to arbitrary modules over $\cN\Gamma$ by L\"{u}ck \cite{Luck98} in order to generalize Atiyah's $L^2$-Betti numbers \cite{Atiyah} to arbitrary continuous $\Gamma$-actions. It has profound application to the theory of $L^2$-invariants \cite{Luck}. Via taking tensor product with $\cN\Gamma$, one can define the von Neumann-L\"{u}ck rank for any module over the integral group ring $\Zb\Gamma$ of $\Gamma$.

Despite the fact that mean dimension and von Neumann-L\"{u}ck rank are invariants in totally different areas, one in dynamical systems and the other in  the theory of $L^2$-invariants,
in this paper we establish a connection between them. This paper should be thought of as a sequel to \cite{LT}, in which a similar connection is established between entropy and $L^2$-torsion. The connection is done via studying an arbitrary $\Zb\Gamma$-module $\cM$ for a discrete amenable group $\Gamma$.
On the one hand, we have the induced $\Gamma$-action on the Pontryagin dual $\widehat{\cM}$ by continuous automorphisms, for which the mean dimension $\mdim(\widehat{\cM})$ is defined (see Section~\ref{SS-mdim}). On the other hand, one has the von Neumann-L\"{u}ck rank $\vrank(\cM)$ defined (see Section~\ref{SS-vrank}). The connection is that $\mdim(\widehat{\cM})=\vrank(\cM)$. In order to prove this, we introduce an invariant $\mrank(\cM)$ of $\cM$, called the mean rank, as an analogue of the rank of a discrete abelian group (see Section~\ref{S-mean rank}). Then we can state our main result as follows:

\begin{theorem} \label{T-main}
For any discrete amenable group $\Gamma$ and any (left) $\Zb\Gamma$-module $\cM$, one has
$$ \mdim(\widehat{\cM})=\mrank(\cM)=\vrank(\cM).$$
\end{theorem}

Theorem~\ref{T-main} has applications to both mean dimension and von Neumann-L\"{u}ck rank. Unlike entropy, in general mean dimension does not necessarily decrease when passing to factors. Using the addition formula either for mean rank or for von Neumann-L\"{u}ck rank, in Corollary~\ref{C-addition mdim} we establish addition formula for mean dimension of {\it algebraic actions}, i.e. actions on compact metrizable abelian groups by continuous automorphisms. In particular, mean dimension does decrease when passing to algebraic factors.

Another application concerns the analogue of the Pontryagin-Schnirelmann theorem for algebraic actions. The Pontryagin-Schnirelmann theorem says that for any compact metrizable space $X$, its covering dimension is the minimal value of Minkowski dimension of $(X, \rho)$ for $\rho$ ranging over compatible metrics on $X$ \cite{PS}. While mean dimension is the dynamical analogue of the covering dimension, Lindenstrauss and Weiss also introduced a dynamical analogue of the Minkowski dimension, called metric mean dimension \cite{LW}. Thus it is natural to ask for a dynamical analogue of the Pontryagin-Schnirelmann theorem. Indeed, Lindentrauss and Weiss showed that mean dimension is always bounded above by metric mean dimension \cite{LW}. Lindenstrauss managed to obtain the full dynamical analogue of the Pontryagin-Schnirelmann theorem, under the condition that $\Gamma=\Zb$ and the action $\Zb\curvearrowright X$ has an infinite minimal factor \cite{Lindenstrauss}. In Theorem~\ref{T-metric mdim} we establish the analogue of the Pontryagin-Schnirelmann theorem for all algebraic actions of countable amenable groups.

The deepest applications of Theorem~\ref{T-main} to mean dimension use the strong Atiyah conjecture \cite{Luck}, which describes possible values of the von Neumann dimension of kernels of matrices over the complex group algebra of a discrete group. Though in general the strong Atiyah conjecture fails \cite{Austin, DS, Grabowski1, Grabowski2, GLSZ, GZ, LWa, PSZ}, it has been verified for large classes of groups \cite{Linnell, Luck}. Using the known cases of the strong Atiyah conjecture, in Corollary~\ref{C-mdim rational} we show that for any elementary amenable group with an upper bound on the orders of the finite subgroups, the range of  mean dimension of its algebraic actions is quite restricted. This is parallel to the result of Lind, Schmidt and Ward  that the range of entropy of algebraic actions of $\Zb^d$ depends on the (still unknown) answer to Lehmer's problem \cite[Theorem 4.6]{LSW}.

The last application we give to mean dimension concerns the structure of algebraic actions with zero mean dimension. Lindestrauss showed that inverse limits of actions with finite topological entropy have zero mean dimension \cite{Lindenstrauss}. He raised implicitly the question whether the converse holds, and showed that this is the case for $\Zb$-actions with infinite minimal factors \cite{Lindenstrauss}. Using the range of mean dimension of algebraic actions, in Corollary~\ref{C-zero} we show that the converse holds for algebraic actions of any elementary amenable group with an upper bound on the orders of the finite subgroups.

We remark that recently mean dimension has been extended to actions of sofic groups \cite{Li13}. It will be interesting to find out whether the equality between
$\mdim(\widehat{\cM})$ and $\vrank(\cM)$ in Theorem~\ref{T-main} holds in sofic case.

This paper is organized as follows. We recall some basic definitions and results in Section~\ref{S-preliminary}. We define mean rank and prove addition formula for it in Section~\ref{S-mean rank}. The equality between $\mdim(\widehat{\cM})$ and $\mrank(\cM)$ is proved in Section~\ref{S-mdim vs mrank}, while the equality between $\mrank(\cM)$ and $\vrank(\cM)$ is proved in Section~\ref{S-mrank vs vrank}.
The rest of the paper is concerned with various applications of Theorem~\ref{T-main}. The addition formula for mean dimension and some vanishing result for von Neumann-L\"{u}ck rank are proved in Section~\ref{S-applications}. The analogue of the Pontryagin-Schnirelmann theorem for algebraic actions is proved in Section~\ref{S-metric mdim}. We discuss the range of mean dimension of algebraic actions in Section~\ref{S-range} and the structure of algebraic actions with zero mean dimension in Section~\ref{S-zero}.

After the first draft of this paper is circulated, G\'{a}bor Elek pointed out to us that the mean rank of a $\Zb\Gamma$-module is closely related to the rank he introduced for finitely generated $\Qb\Gamma$-modules in \cite{Elek1}. In  Appendix~\ref{S-Elek} we discuss the precise relation between these two ranks.

\noindent{\it Acknowledgements.}
H. Li was partially supported by NSF grants DMS-1001625 and DMS-1266237. He thanks Masaki Tsukamoto for bringing his attention to Question~\ref{Q-zero}.
We are grateful to Nhan-Phu Chung, Yonatan Gutman, David Kerr and Andreas Thom for helpful comments, and to G\'{a}bor Elek for alerting us to  \cite{Elek1}.

\section{Preliminaries} \label{S-preliminary}

For any set $X$ we denote by $\cF(X)$ the set of all nonempty finite subsets of $X$. When a group $\Gamma$ acts on a set $X$, for any $F\subseteq \Gamma$ and $A\subseteq X$ we denote $\bigcup_{s\in F}sA$ by $FA$.


\subsection{Group rings} \label{SS-group rings}

Let $\Gamma$ be a discrete group with identity element $e$. The {\it integral group ring} of $\Gamma$, denoted by $\Zb\Gamma$, consists of all finitely supported functions $f: \Gamma\rightarrow \Zb$. We shall write $f$ as $\sum_{s\in \Gamma}f_s s$. The ring structure of $\Zb\Gamma$ is defined by
$$ \sum_{s\in \Gamma}f_ss+\sum_{s\in \Gamma} g_ss=\sum_{s\in \Gamma}(f_s+g_s)s, \, \, \, \, \quad (\sum_{s\in \Gamma}f_ss)\cdot (\sum_{s\in \Gamma} g_ss)=\sum_{s\in \Gamma}(\sum_{t\in \Gamma}f_tg_{t^{-1}s})s.$$
Similarly, one defines $\Cb\Gamma$.

Denote by $\ell^2(\Gamma)$ the Hilbert space of all functions $x: \Gamma\rightarrow \Cb$ satisfying $\sum_{s\in \Gamma}|x_s|^2<+\infty$.
The {\it left regular representation} and {\it right regular representation} of $\Gamma$ on $\ell^2(\Gamma)$ are defined by
$$ (l_sx)_t=x_{s^{-1}t} \mbox{ and } (r_sx)_t=x_{ts}$$
respectively, and commute with each other. The {\it group von Neumann algebra} $\cN\Gamma$ of $\Gamma$ is defined as the $*$-algebra of all bounded linear operators on $\ell^2(\Gamma)$  commuting with the image of the right regular representation. See \cite[Section V.7]{Takesaki} for detail.
Via the left regular representation, we may identify $\Zb\Gamma$ with a subring of $\cN\Gamma$.

For $m, n\in \Nb$, we shall think of elements of $M_{n, m}(\cN\Gamma)$ as bounded linear operators from $(\ell^2(\Gamma))^{m\times 1}$ to  $(\ell^2(\Gamma))^{n\times 1}$. There is a canonical trace $\tr_{\cN\Gamma}$ on $M_n(\cN\Gamma)$ defined by
$$\tr_{\cN\Gamma}f=\sum_{j=1}^n\left<f_{j, j}e, e\right>$$
for $f=(f_{j, k})_{1\le j, k\le n}\in M_n(\cN\Gamma)$, where via the natural embedding $\Gamma\hookrightarrow \Cb\Gamma\hookrightarrow \ell^2(\Gamma)$ we identify $\Gamma$ with the canonical orthonormal basis of $\ell^2(\Gamma)$. For any $f\in M_{m,n}(\cN\Gamma)$ and $g\in M_{n, m}(\cN\Gamma)$, one has the tracial property
\begin{align} \label{E-trace}
\tr_{\cN\Gamma}(fg)=\tr_{\cN\Gamma}(gf)
\end{align}

\subsection{von Neumann-L\"{u}ck rank} \label{SS-vrank}

For a finitely generated projective (left) $\cN\Gamma$-module $\fM$, take $P\in M_n(\cN\Gamma)$ for some $n\in \Nb$ such that $P^2=P$ and
$(\cN\Gamma)^{1\times n}P$ is isomorphic to $\fM$ as $\cN\Gamma$-modules. The {\it von Neumann dimension} $\dim_{\cN\Gamma}' \fM$ of $\fM$
is defined as
$$\dim_{\cN\Gamma}'\fM=\tr_{\cN\Gamma}P\in [0, n],$$
and does not depend on the choice of $P$. For an arbitrary (left) $\cN\Gamma$-module $\fM$, its {\it von Neumann-L\"{u}ck dimension} $\dim_{\cN\Gamma}\fM$ \cite[Definition 6.6]{Luck} is defined as
$$\dim_{\cN\Gamma}\fM=\sup_{\fN}\dim_{\cN\Gamma}'\fN\in [0, +\infty],$$
where $\fN$ ranges over all finitely generated projective submodules of $\fM$.

We collect a few fundamental properties of the von Neumann-L\"{u}ck dimension here \cite[Theorem 6.7]{Luck}:

\begin{theorem} \label{T-dim basic}
The following hold.
\begin{enumerate}
\item For any short exact sequence
$$0\rightarrow \fM_1\rightarrow \fM_2\rightarrow \fM_3\rightarrow 0$$
of $\cN\Gamma$-modules, one has
$$\dim_{\cN\Gamma}\fM_2=\dim_{\cN\Gamma}\fM_1+\dim_{\cN\Gamma}\fM_3.$$
\item For any $\cN\Gamma$-module $\fM$ and any increasing net $\{\fM_n\}_{n\in J}$ of $\cN\Gamma$-submodules of $\fM$ with union $\fM$, one has
$$ \dim_{\cN\Gamma}\fM=\sup_{n\in J}\dim_{\cN\Gamma}\fM_n.$$
\item $\dim_{\cN\Gamma}$ extends $\dim_{\cN\Gamma}'$, i.e. for any $n\in \Nb$ and any  idempotent $P\in M_n(\cN\Gamma)$, one
has $\dim_{\cN\Gamma}((\cN\Gamma)^{1\times n}P)=\tr_{\cN\Gamma}P$. In particular, $\dim_{\cN\Gamma}\cN\Gamma=1$.
\end{enumerate}
\end{theorem}

For any (left) $\Zb\Gamma$-module $\cM$, its {\it von Neumann-L\"{u}ck rank} $\vrank(\cM)$ is defined by
$$ \vrank(\cM):=\dim_{\cN\Gamma}(\cN\Gamma\otimes_{\Zb\Gamma}\cM).$$

\subsection{Amenable group} \label{SS-amenable group}

Let $\Gamma$ be a discrete group. For $K\in \cF(\Gamma)$ and $\delta>0$, denote by $\cB(K, \delta)$ the set of all $F\in \cF(\Gamma)$ satisfying $|KF\setminus F|<\delta |F|$. The group $\Gamma$ is called {\it amenable} if $\cB(K, \delta)$ is nonempty for every $(K, \delta)$ \cite[Section 4.9]{CC}.

The collection of pairs $(K, \delta)$ forms a net $\Lambda$ where $(K', \delta')\succ(K, \delta)$ means $K'\supseteq K$ and $\delta'<\delta$. For a
$\Rb$-valued function $\varphi$ defined on $\cF(\Gamma)$, we say that $\varphi(F)$ converges to $c\in \Rb$ when $F\in \cF(\Gamma)$ becomes more and more left
invariant, denoted by $\lim_F\varphi(F)=c$, if for any $\varepsilon>0$ there is some $(K, \delta)\in \Lambda$ such that  $|\varphi(F)-c|<\varepsilon$ for all $F\in \cB(K, \delta)$. In general, we define
$$ \varlimsup_F\varphi(F):=\lim_{(K, \delta)\in \Lambda}\sup_{F\in \cB(K, \delta)}\varphi(F). $$

\subsection{Mean dimension} \label{SS-mdim}

Let $X$ be a compact Hausdorff space. For two finite open covers $\cU$ and $\cV$ of $X$, we say that $\cV$ {\it refines} $\cU$ and write $\cV\succ \cU$, if
every item of  $\cV$ is contained in some item of $\cU$.
We set
$$\ord(\cU):=\max_{x\in X}\sum_{U\in \cU}1_U(x)-1,$$
where $1_U$ denotes the characteristic function of $U$, and
$$\cD(\cU):=\min_{\cV\succ \cU}\ord(\cV)$$
for $\cV$ ranging over all finite open covers of $X$ refining $\cU$. The {\it covering dimension} of $X$, denoted by $\dim(X)$, is defined as
the supremum of $\cD(\cU)$ for $\cU$ ranging over finite open covers of $X$ \cite[Section V.8]{HW}.

Consider a continuous action of a discrete amenable group $\Gamma$ on $X$. For any finite open covers $\cU$ and $\cV$ of $X$, the joining $\cU\vee \cV$ is the finite open cover of $X$ consisting of $U\cap V$ for all $U\in \cU$ and $V\in \cV$. For any $F\in \cF(\Gamma)$, set $\cU^F=\bigvee_{s\in F}s^{-1}\cU$. The function
$\cF(\Gamma)\rightarrow \Rb$ sending $F$ to $\cD(\cU^F)$ satisfies the conditions of the Ornstein-Weiss lemma \cite{OW} \cite[Theorem 6.1]{LW}, and hence
the limit $\lim_F\frac{\cD(\cU^F)}{|F|}$ exists, which we denote by $\mdim(\cU)$. The {\it mean dimension} of the action $\Gamma \curvearrowright X$, denoted by $\mdim(X)$, is defined as the supremum of $\mdim(\cU)$ for $\cU$ ranging over all finite open covers of $X$ \cite[Definition 2.6]{LW}.

\section{Mean rank and addition formula} \label{S-mean rank}

Throughout the rest of this paper $\Gamma$ will be a discrete amenable group, unless specified otherwise.

In this section we define mean rank for $\Zb\Gamma$-modules and prove the addition formula for mean rank.

Recall that the {\it rank} of a discrete abelian group $\sM$ is defined as $\dim_\Qb(\Qb\otimes_\Zb\sM)$, which we shall denote by $\rank(\sM)$. We say that a
subset $A$ of $\sM$ is {\it linearly independent} if every finitely supported function $\lambda: A\rightarrow \Zb$ satisfying $\sum_{a\in A}\lambda_aa=0$ in $\sM$ must be $0$. It is clear that the cardinality of any maximal linearly independent subset of $\sM$ is equal to $\rank(\sM)$.

We need the following elementary property about rank a few times:

\begin{lemma} \label{L-rank}
For any short exact sequence
$$0\rightarrow \sM_1\rightarrow \sM_2\rightarrow \sM_3\rightarrow 0$$
of abelian groups, one has $\rank(\sM_2)=\rank(\sM_1)+\rank(\sM_3)$.
\end{lemma}
\begin{proof} Since the functor $\Qb \otimes_\Zb \cdot$ is exact \cite[Proposition XVI.3.2]{Lang}, the sequence
$$ 0\rightarrow \Qb\otimes_\Zb\sM_1\rightarrow \Qb\otimes_\Zb\sM_2\rightarrow \Qb\otimes_\Zb\sM_3\rightarrow 0$$
is exact. Thus
$$ \rank(\sM_2)=\dim_\Qb(\Qb\otimes_\Zb\sM_2)=\dim_\Qb(\Qb\otimes_\Zb\sM_1)+\dim_\Qb(\Qb\otimes_\Zb\sM_3)=\rank(\sM_1)+\rank(\sM_3).$$
\end{proof}

For any discrete abelian group $\sM$ and $A\subseteq \sM$, we denote by $\left<A\right>$ the subgroup of $\sM$ generated by $A$.

Let $\cM$ be a (left) $\Zb\Gamma$-module.

\begin{lemma} \label{L-mrank def}
For any $A\in \cF(\cM)$, one has
$$\lim_F\frac{\rank(\left<F^{-1}A\right>)}{|F|}=\inf_{F\in \cF(\Gamma)}\frac{\rank(\left<F^{-1}A\right>)}{|F|}.$$
\end{lemma}

Lemma~\ref{L-mrank def} follows from Lemma~\ref{L-property of rank} and the fact that for any function $\varphi$ satisfying the three conditions in Lemma~\ref{L-property of rank} one has $\lim_F\frac{\varphi(F)}{|F|}=\inf_{F\in \cF(\Gamma)}\frac{\varphi(F)}{|F|}$ \cite[Definitions 2.2.10 and 3.1.5, Remark 3.1.7, and Proposition 3.1.9]{JMO} \cite[Lemma 3.3]{LT}.

\begin{lemma} \label{L-property of rank}
Let $A\in \cF(\cM)$. Define $\varphi: \cF(\Gamma)\cup\{\emptyset\}\rightarrow \Zb$ by $\varphi(F)=\rank(\left<F^{-1}A\right>)$. Then the following hold:
\begin{enumerate}
\item $\varphi(\emptyset)=0$;

\item $\varphi(Fs)=\varphi(F)$ for all $F\in \cF(\Gamma)$ and $s\in \Gamma$;

\item $\varphi(F_1\cup F_2)+\varphi(F_1\cap F_2)\le \varphi(F_1)+\varphi(F_2)$ for all $F_1, F_2\in \cF(\Gamma)$.
\end{enumerate}
\end{lemma}
\begin{proof} (1) and (2) are trivial. Let $F_1, F_2\in \cF(\Gamma)$.
Set $\sM_j=\left<F_j^{-1}A\right>$ for $j=1, 2$, $\sM=\left<(F_1\cup F_2)^{-1}A\right>$, and
$\sN=\left<(F_1\cap F_2)^{-1}A\right>$. Then $\sM_1+\sM_2=\sM$, and $\sN\subseteq \sM_1\cap \sM_2$.
By Lemma~\ref{L-rank} we have
\begin{align*}
\varphi(F_1\cup F_2)-\varphi(F_1)&=\rank(\sM)-\rank(\sM_1)\\
&=\rank(\sM/\sM_1)\\
&=\rank(\sM_2/\sM_1\cap \sM_2)\\
&=\rank(\sM_2)-\rank(\sM_1\cap \sM_2)\\
&\le \rank(\sM_2)-\rank(\sN)\\
&=\varphi(F_2)-\varphi(F_1\cap F_2).
\end{align*}
\end{proof}

\begin{remark} \label{R-mrank}
One can also prove the existence of the limit $\lim_F\frac{\rank(\left<F^{-1}A\right>)}{|F|}$ for every $A\in \cF(\cM)$ using the Ornstein-Weiss lemma \cite{OW} \cite[Theorem 6.1]{LW}. The fact that this limit is equal to $\inf_{F\in \cF(\Gamma)}\frac{\rank(\left<F^{-1}A\right>)}{|F|}$ will be crucial in the proof of Lemma~\ref{L-decreasing limit mrank} below which discusses the behavior of mean rank under taking decreasing direct limit.
\end{remark}

\begin{definition} \label{D-mean rank}
We define the {\it mean rank} of a (left) $\Zb\Gamma$-module $\cM$ as
$$\mrank(\cM):=\sup_{A\in \cF(\cM)}\lim_F \frac{\rank(\left<F^{-1}A\right>)}{|F|}=\sup_{A\in \cF(\cM)}\inf_{F\in \cF(\Gamma)} \frac{\rank(\left<F^{-1}A\right>)}{|F|}.$$
\end{definition}

The main result of this section is the following addition formula for mean rank under taking extensions of $\Zb\Gamma$-modules.

\begin{theorem} \label{T-addition mrank}
For any short exact sequence
$$0\rightarrow \cM_1\rightarrow \cM_2\rightarrow \cM_3\rightarrow 0$$
of $\Zb\Gamma$-modules, we have
$$\mrank(\cM_2)=\mrank(\cM_1)+\mrank(\cM_3).$$
\end{theorem}

To prove Theorem~\ref{T-addition mrank} we need some preparation.

\begin{lemma} \label{L-addition}

Let
$$0\rightarrow \cM_1\rightarrow \cM_2\rightarrow \cM_3\rightarrow 0$$
be a short exact sequence of $\Zb\Gamma$-modules. Then
$$ \mrank(\cM_2)\ge \mrank(\cM_1)+\mrank(\cM_3).$$
\end{lemma}
\begin{proof} Denote by $\pi$ the homomorphism $\cM_2\rightarrow \cM_3$. We shall think of $\cM_1$ as a submodule of $\cM_2$.
It suffices to show that for any $A_1\in \cF(\cM_1)$ and any $A_3\in \cF(\cM_3)$, taking $A_3'\in \cF(\cM_2)$ with $\pi(A_3')=A_3$ and setting $A_2=A_1\cup A_3'\in \cF(\cM_2)$ one has
$$\lim_F \frac{\rank(\left<F^{-1}A_2\right>)}{|F|}\ge \lim_F \frac{\rank(\left<F^{-1}A_1\right>)}{|F|}+\lim_F \frac{\rank(\left<F^{-1}A_3\right>)}{|F|}.$$
In turn it suffices to show that for any $F\in \cF(\Gamma)$, one has
\begin{align} \label{E-addition}
\rank(\left<F^{-1}A_2\right>)\ge \rank(\left<F^{-1}A_1\right>)+\rank(\left<F^{-1}A_3\right>).
\end{align}
Set $\sM_j=\left<F^{-1}A_j\right>$ for $j=1, 2, 3$. Then $\pi(\sM_2)=\sM_3$ and $\sM_1\subseteq \cM_1\cap \sM_2$. From the short exact sequence
$$0\rightarrow \cM_1\cap \sM_2\rightarrow \sM_2\rightarrow \sM_3\rightarrow 0$$
of abelian groups, by Lemma~\ref{L-rank}
we have
\begin{align*}
\rank(\sM_2)=\rank(\cM_1\cap \sM_2)+\rank(\sM_3)\ge \rank(\sM_1)+\rank(\sM_3),
\end{align*}
yielding \eqref{E-addition}.
\end{proof}

\begin{lemma} \label{L-addition2}
Theorem~\ref{T-addition mrank} holds when
$\cM_2$ is a submodule of $(\Zb\Gamma)^n$ for some $n\in \Nb$ and $\cM_1$ is finitely generated.
\end{lemma}
\begin{proof} By Lemma~\ref{L-addition} it suffices to show that $\mrank(\cM_2)\le \mrank(\cM_1)+\mrank(\cM_3)$.
Denote by $\pi$ the homomorphism $\cM_2\rightarrow \cM_3$. We shall think of $\cM_1$ as a submodule of $\cM_2$. Take a finite generating subset $A_1$ of $\cM_1$. Then it suffices to show that
for any $A_2\in \cF(\cM_2)$, setting $A_3=\pi(A_2)\in \cF(\cM_3)$ one has
$$\lim_F \frac{\rank(\left<F^{-1}A_2\right>)}{|F|}\le \lim_F \frac{\rank(\left<F^{-1}A_1\right>)}{|F|}+\lim_F \frac{\rank(\left<F^{-1}A_3\right>)}{|F|}.$$
Replacing $A_2$ by $A_1\cup A_2$ if necessary, we may assume that $A_1\subseteq A_2$.

Let $F\in \cF(\Gamma)$.
Set $\sM_j=\left<F^{-1}A_j\right>$ for $j=1, 2, 3$. Then we have a short exact sequence
$$0\rightarrow \cM_1\cap \sM_2\rightarrow \sM_2\rightarrow \sM_3\rightarrow 0$$
of abelian groups. Denote by $K_j$ the union of supports of elements in $A_j$ as $\Zb^n$-valued functions on $\Gamma$ for $j=1, 2$.
Then $K_j$ is a finite subset of $\Gamma$ and elements of $\sM_j$ have support contained in $F^{-1}K_j$.
Note that every $x\in \cM_1$ can be written as $y_x+z_x$ for some $y_x\in \sM_1$ and some $z_x\in \left<(\Gamma\setminus F)^{-1}A_1\right>$.
The support of $z_x$ is contained in $(\Gamma\setminus F)^{-1}K_1$. Therefore, for any $x\in \cM_1\cap \sM_2$, the support of $z_x$ is contained in
$(\Gamma\setminus F)^{-1}K_1\cap F^{-1}(K_1\cup K_2)=(\Gamma\setminus F)^{-1}K_1\cap F^{-1}K_2$. It follows that
$$\rank((\cM_1\cap \sM_2)/\sM_1)\le n|(\Gamma\setminus F)^{-1}K_1\cap F^{-1}K_2|=n|K_1^{-1}(\Gamma\setminus F)\cap K_2^{-1}F|.$$
Thus, by Lemma~\ref{L-rank} we get
\begin{align*}
\rank(\sM_2)
&=\rank(\cM_1\cap \sM_2)+\rank(\sM_3)\\
&=\rank((\cM_1\cap \sM_2)/\sM_1)+\rank(\sM_1)+\rank(\sM_3)\\
&\le n|K_1^{-1}(\Gamma\setminus F)\cap K_2^{-1}F|+\rank(\sM_1)+\rank(\sM_3).
\end{align*}
Consequently,
\begin{eqnarray*}
& &\lim_F \frac{\rank(\left<F^{-1}A_2\right>)}{|F|}-\lim_F \frac{\rank(\left<F^{-1}A_1\right>)}{|F|}-\lim_F \frac{\rank(\left<F^{-1}A_3\right>)}{|F|}\\
&\le& \varlimsup_F\frac{n|K_1^{-1}(\Gamma\setminus F)\cap K_2^{-1}F|}{|F|}=0.
\end{eqnarray*}
\end{proof}

The following lemma is trivial, discussing the behavior of mean rank  under taking
increasing union of $\Zb\Gamma$-modules.

\begin{lemma} \label{L-increasing limit mrank}
Let $\cM$ be a $\Zb\Gamma$-module and $\{\cM_n\}_{n\in J}$ be an increasing net of submodules of $\cM$ with union $\cM$. Then $$\mrank(\cM)=\lim_{n\to\infty}\mrank(\cM_n)=\sup_{n\in J}\mrank(\cM_n).$$
\end{lemma}

The next lemma says that when $\cM$ is finitely generated, to compute the mean rank, it is enough to do calculation for one finite generating set.

\begin{lemma} \label{L-finitely generated}
Let $\cM$ be a finitely generated $\Zb\Gamma$-module with a finite generating subset $A$. Then
$$ \mrank(\cM)=\lim_F \frac{\rank(\left<F^{-1}A\right>)}{|F|}.$$
\end{lemma}
\begin{proof} It suffices to show
$$\lim_F \frac{\rank(\left<F^{-1}A'\right>)}{|F|}\le \lim_F \frac{\rank(\left<F^{-1}A\right>)}{|F|}$$
for every $A'\in \cF(\cM)$. We have $A'\subseteq \left<K^{-1}A\right>$ for some $K\in \cF(\Gamma)$.
For any $F\in \cF(\Gamma)$, we have
$$\left<F^{-1}A'\right>\subseteq \left<(KF)^{-1}A\right>.$$
Thus
\begin{align*}
\lim_F \frac{\rank(\left<F^{-1}A'\right>)}{|F|}
&\le \lim_F \frac{\rank(\left<(KF)^{-1}A\right>)}{|F|}\\
&=\lim_F \frac{\rank(\left<(KF)^{-1}A\right>)}{|KF|}\\
&=\lim_F \frac{\rank(\left<F^{-1}A\right>)}{|F|}.
\end{align*}
\end{proof}

The next lemma discusses the behavior of mean rank under taking decreasing direct limit for finitely generated $\Zb\Gamma$-modules.

\begin{lemma} \label{L-decreasing limit mrank}
Let $\cM$ be a finitely generated $\Zb\Gamma$-module. Let $\{\cM_n\}_{n\in J}$ be an increasing net of submodules of $\cM$ with $\infty\not\in J$. Set $\cM_\infty=\bigcup_{n\in J}\cM_n$. Then
$$\mrank(\cM/\cM_\infty)=\lim_{n\to\infty}\mrank(\cM/\cM_n)=\inf_{n\in J} \mrank(\cM/\cM_n).$$
\end{lemma}
\begin{proof} By Lemma~\ref{L-addition} we have $\mrank(\cM/\cM_n)\ge \mrank(\cM/\cM_{m})\ge \mrank(\cM/\cM_\infty)$ for all $n, m\in J$ with $m\ge n$.
Thus
$$\lim_{n\to\infty}\mrank(\cM/\cM_n)=\inf_{n\in J} \mrank(\cM/\cM_n)\ge \mrank(\cM/\cM_\infty).$$

Denote by $\pi_n$ the quotient map $\cM\rightarrow \cM/\cM_n$ for $n\in J\cup\{\infty\}$. Let $A$ be a finite generating subset of $\cM$.
By Lemmas~\ref{L-finitely generated} and \ref{L-mrank def} we have
$$ \mrank(\cM/\cM_n)=\inf_{F\in \cF(\Gamma)}\frac{\rank(\left<F^{-1}\pi_n(A)\right>)}{|F|}$$
for all $n\in J\cup\{\infty\}$.

Let $F\in \cF(\Gamma)$.
Note that $\cM_\infty\cap \left<F^{-1}A\right>=\bigcup_{n\in J}(\cM_n\cap \left<F^{-1}A\right>)$.
Thus
$$\rank(\cM_n\cap \left<F^{-1}A\right>)\rightarrow \rank(\cM_\infty\cap \left<F^{-1}A\right>)$$
as $n\to \infty$.
By Lemma~\ref{L-rank} we get
$$\rank(\left<F^{-1}\pi_n(A)\right>)\rightarrow \rank(\left<F^{-1}\pi_\infty(A)\right>)$$
as $n\to \infty$.
Thus
$$ \lim_{n\to\infty}\mrank(\cM/\cM_n)\le \lim_{n\to \infty}\frac{\rank(\left<F^{-1}\pi_n(A)\right>)}{|F|}=\frac{\rank(\left<F^{-1}\pi_\infty(A)\right>)}{|F|}.$$
Taking infimum over $F\in \cF(\Gamma)$, we get
$\lim_{n\to\infty}\mrank(\cM/\cM_n)\le \mrank(\cM/\cM_\infty)$.
\end{proof}

\begin{remark} \label{R-desreasing limit}
Lemma~\ref{L-decreasing limit mrank} does not hold for  arbitrary $\Zb\Gamma$-module $\cM$. For example, take $\Gamma$ to be the trivial group, $\cM=\bigoplus_{j\in \Nb}\Zb$, and $\cM_n=\bigoplus_{1\le j\le n}\Zb$ for all $n\in \Nb$. Then $\mrank(\cM/\cM_n)=\infty$ for all $n\in \Nb$ while $\mrank(\cM/\cM_\infty)=0$.
\end{remark}

\begin{lemma} \label{L-addition3}
Theorem~\ref{T-addition mrank} holds when
$\cM_2$ is a finitely generated submodule of $(\Zb\Gamma)^n$ for some $n\in \Nb$.
\end{lemma}
\begin{proof} We shall think of $\cM_1$ as a submodule of $\cM_2$. Take an increasing net of finitely generated submodules $\{\cM_j\}_{j\in J}$ of $\cM_1$ such that $\{1, 2, 3\}\cap J=\emptyset$ and $\bigcup_{j\in J}\cM_j=\cM_1$.
From Lemmas~\ref{L-increasing limit mrank} and \ref{L-decreasing limit mrank} we have
$$\mrank(\cM_1)=\lim_{j\to \infty}\mrank(\cM_j)$$
and
$$ \mrank(\cM_3)=\mrank(\cM_2/\cM_1)=\lim_{j\to \infty}\mrank(\cM_2/\cM_j).$$
By Lemma~\ref{L-addition2} we have
\begin{align*}
\mrank(\cM_2)=\mrank(\cM_j)+\mrank(\cM_2/\cM_j)
\end{align*}
for every $j\in J$. Letting $j\to \infty$, we obtain $\mrank(\cM_2)=\mrank(\cM_1)+\mrank(\cM_3)$.
\end{proof}

\begin{lemma} \label{L-addition4}
Theorem~\ref{T-addition mrank} holds when
$\cM_2$ is a submodule of $(\Zb\Gamma)^n$ for some $n\in \Nb$.
\end{lemma}
\begin{proof} We shall think of $\cM_1$ as a submodule of $\cM_2$.
Take an increasing net of finitely generated submodules $\{\cM_j\}_{j\in J}$ of $\cM_2$ such that $\{1, 2, 3\}\cap J=\emptyset$ and $\bigcup_{j\in J}\cM_j=\cM_2$.
From Lemma~\ref{L-increasing limit mrank}  we have
$$\mrank(\cM_2)=\lim_{j\to \infty}\mrank(\cM_j),$$
and
$$\mrank(\cM_1)=\lim_{j\to \infty} \mrank(\cM_j\cap \cM_1), $$
and
$$ \mrank(\cM_3)=\mrank(\cM_2/\cM_1)=\lim_{j\to \infty}\mrank((\cM_j+\cM_1)/\cM_1)=\lim_{j\to \infty}\mrank(\cM_j/(\cM_j\cap \cM_1)).$$
By Lemma~\ref{L-addition3} we have
\begin{align*}
\mrank(\cM_j)=\mrank(\cM_j\cap \cM_1)+\mrank(\cM_j/(\cM_j\cap \cM_1))
\end{align*}
for every $j\in J$. Letting $j\to \infty$, we obtain $\mrank(\cM_2)=\mrank(\cM_1)+\mrank(\cM_3)$.
\end{proof}

\begin{lemma} \label{L-addition5}
Theorem~\ref{T-addition mrank} holds when
$\cM_2$ is  a finitely generated $\Zb\Gamma$-module.
\end{lemma}
\begin{proof} We shall think of $\cM_1$ as a submodule of $\cM_2$. Take a surjective $\Zb\Gamma$-module homomorphism $\pi: (\Zb\Gamma)^n\rightarrow \cM_2$ for some $n\in \Nb$.
By Lemma~\ref{L-addition4} we have
$$ \mrank((\Zb\Gamma)^n)=\mrank(\ker \pi)+\mrank(\cM_2),$$
and
\begin{align*}
\mrank((\Zb\Gamma)^n)&=\mrank(\pi^{-1}(\cM_1))+\mrank((\Zb\Gamma)^n/\pi^{-1}(\cM_1))\\
&=\mrank(\pi^{-1}(\cM_1))+\mrank(\cM_2/\cM_1)\\
&=\mrank(\pi^{-1}(\cM_1))+\mrank(\cM_3),
\end{align*}
and
$$\mrank(\pi^{-1}(\cM_1))=\mrank(\ker \pi)+\mrank(\cM_1).$$
Therefore
$$\mrank(\ker \pi)+\mrank(\cM_1)+\mrank(\cM_3)=\mrank(\ker \pi)+\mrank(\cM_2).$$
By Lemmas~\ref{L-addition} and \ref{L-finitely generated} we have $\mrank(\ker \pi)\le \mrank((\Zb\Gamma)^n)<+\infty$.
It follows that $\mrank(\cM_1)+\mrank(\cM_3)=\mrank(\cM_2)$.
\end{proof}

Finally, replacing Lemma~\ref{L-addition3} in the proof of Lemma~\ref{L-addition4} by Lemma~\ref{L-addition5}, we obtain Theorem~\ref{T-addition mrank} in full generality.

\begin{remark} \label{R-SVV}
For a unital ring $R$, a {\it length function} $L$ on (left) $R$-modules \cite{NR} means associating a value $L(\sM)\in \Rb_{\ge 0}\cup \{+\infty\}$ for each $R$-module $\sM$ such that $L(0)=0$, additivity holds for short exact sequences of $R$-modules (i.e. the analogue of Lemma~\ref{L-rank}), and $L(\sM)$ is equal to the supremum of $L(\sM')$ for $\sM'$ ranging over finitely generated $R$-submodules of $\sM$. If $L$ is a length function on $R$-modules such that $L(R)<+\infty$, then one can define a mean length on (left) $R\Gamma$-modules for any discrete amenable group $\Gamma$, and all the results in this section including Theorem~\ref{T-addition mrank} still hold without change of proof.

A length function is called {\it discrete} if the set of its finite values is order-isomorphic to $\Nb$. In \cite{SVV}, given a length function $L$ on $R$-modules, Salce, V\'{a}mos and Virili defined an invariant ${\rm ent}_L$ on $R\Zb$-modules whose finitely generated $R$-submodules all take finite $L$-values, in exactly the same way as we define the mean rank in Definition~\ref{D-mean rank}. They also proved Theorem~\ref{T-addition mrank} for $R\Zb$-modules whose finitely generated $R$-submodules take finite $L$-values, under the condition that $L$ is discrete. They do not assume $L(R)<+\infty$ as we do here, and there are interesting examples of length functions satisfying $L(R)=+\infty$. On the other hand, they need $\Gamma=\Zb$ and $L$ to be discrete, and their proof of the addition formula (Theorem~\ref{T-addition mrank}) uses both of these two conditions in a fundamental way.
\end{remark}

\section{Mean dimension and mean rank} \label{S-mdim vs mrank}

Throughout the rest of this article, for any discrete abelian group $\sM$, we denote by $\widehat{\sM}$ the Pontryagin dual of $\sM$, which is a compact Hausdorff abelian group consisting of all group homomorphisms $\sM\rightarrow \Rb/\Zb$. For any $\Zb\Gamma$-module $\cM$, the module structure of $\cM$ gives rise to an action of
$\Gamma$ on the discrete abelian group $\cM$ by group homomorphisms, which in turn gives rise to an action of $\Gamma$ on $\widehat{\cM}$ by continuous group homomorphisms. Explicitly, for any $a\in \cM$, $x\in \widehat{\cM}$, and $s\in \Gamma$, one has
$$ (sx)(a)=x(s^{-1}a).$$
By Pontryagin duality, every algebraic action of $\Gamma$, i.e. an action of $\Gamma$ on a compact abelian group by continuous group homomorphisms, is of the form
$\Gamma\curvearrowright\widehat{\cM}$ for some $\Zb\Gamma$-module $\cM$.

In this section we prove the following

\begin{theorem} \label{T-mdim vs mrank}
For any $\Zb\Gamma$-module $\cM$, one has $\mdim(\widehat{\cM})=\mrank(\cM)$.
\end{theorem}

Peters gave a formula computing the entropy of $\Gamma\curvearrowright \widehat{\cM}$ in terms of the data of $\cM$ \cite[Theorem 6]{Peters} \cite[Theorem 4.10]{LT}, which plays a crucial role in recent study of the entropy of algebraic actions \cite{CL, LT}. Theorem~\ref{T-mdim vs mrank} is an analogue of Peters' formula for computing the mean dimension of $\Gamma\curvearrowright \widehat{\cM}$ in terms of the data of $\cM$. When $\Gamma$ is the trivial group,
one recovers the classical result of Pontryagin that for any  discrete abelian group $\sM$ one has $\dim(\widehat{\sM})=\rank(\sM)$ \cite[page 259]{Pontryagin}.

Theorem~\ref{T-mdim vs mrank} follows from Lemmas~\ref{L-mdim>mrank} and \ref{L-mdim<mrank} below.

The proof of the following lemma is inspired by the argument in \cite[page 259]{Pontryagin}.

\begin{lemma} \label{L-mdim>mrank}
For any $\Zb\Gamma$-module $\cM$, one has $ \mdim(\widehat{\cM})\ge \mrank(\cM)$.
\end{lemma}
\begin{proof} It suffices to show that for any $A\in \cF(\cM)$, there is some finite open cover $\cU$ of $\widehat{\cM}$ such that
$$\lim_F \frac{\cD(\cU^F)}{|F|}\ge \lim_F \frac{\rank(\left<F^{-1}A\right>)}{|F|}.$$
Take a finite open cover $\cU$ of $\widehat{\cM}$ such that for any $a\in A$, no item $U$ of $\cU$ intersects both $a^{-1}(\Zb)$ and $a^{-1}(1/2+\Zb)$.
Then it suffices to show that for every $F\in \cF(\Gamma)$, one has $\cD(\cU^F)\ge \rank(\left<F^{-1}A\right>)$.

Take a maximal linearly independent subset $B$ of  $F^{-1}A$. Then $|B|=\rank(\left<F^{-1}A\right>)$. Consider the natural abelian group homomorphism $\varphi: \cM\rightarrow \Qb\otimes_\Zb\cM$ sending $a$ to $1\otimes a$. Then $\varphi$ is injective on $B$,  and $\varphi(B)$ is linear independent. Denote by $W$ the $\Qb$-linear span of $\varphi(B)$. By taking a basis of the $\Qb$-vector space $\Qb\otimes_\Zb\cM$ containing $\varphi(B)$, we can find a $\Qb$-linear map $\psi: \Qb\otimes_\Zb\cM \rightarrow W$ being the identity map on $W$.

Now we define an embedding $\iota$ from $[0, 1/2]^B$ into $\widehat{\cM}$ as follows. For each $\lambda=(\lambda_b)_{b\in B}\in [0, 1/2]^B$, we define
a $\Qb$-linear map $g_\lambda: W\rightarrow \Rb$ sending $\varphi(b)$ to $\lambda_b$ for all $b\in B$, and
an abelian group homomorphism
$\iota_\lambda: \cM\rightarrow \Rb/\Zb$ sending $a$ to $g_\lambda(\psi(\varphi(a)))+\Zb$. Then $\iota_\lambda \in \widehat{\cM}$. Clearly the map $\iota: [0, 1/2]^B\rightarrow \widehat{\cM}$ sending $\lambda$ to $\iota_\lambda$ is continuous. Note that $\iota_\lambda(b)=\lambda_b+\Zb$ for all $b\in B$. Thus $\iota$ is injective and hence is an embedding.

The pull back $\iota^{-1}(\cU^F)$ is a finite open cover of $[0, 1/2]^B$.
We claim that no item of $\iota^{-1}(\cU^F)$ intersects two opposing faces of the cube $[0, 1/2]^B$. Suppose that some item $\iota^{-1}(V)$ of $\iota^{-1}(\cU^F)$ contains two points $\lambda$ and $\lambda'$ in opposing faces of $[0, 1/2]^B$, where $V$ is an item of $\cU^F$. Say, $V=\bigcap_{s\in F}s^{-1}U_s$ with $U_s\in \cU$ for each $s\in F$, and $\lambda_{b_0}=0$ and $\lambda'_{b_0}=1/2$ for some $b_0\in B$.  Then $b_0=s_0^{-1}a_0$ for some $s_0\in F$ and $a_0\in A$. Now we have
\begin{align*}
(s_0\iota_\lambda)(a_0)=\iota_\lambda(s_0^{-1}a_0)=\iota_\lambda(b_0)=\lambda_{b_0}+\Zb=\Zb,
\end{align*}
and similarly $(s_0\iota_{\lambda'})(a_0)=1/2+\Zb$. Since $\lambda \in \iota^{-1}(V)$, we have $\iota_\lambda\in V\subseteq s_0^{-1}U_{s_0}$, and hence
$s_0\iota_\lambda\in U_{s_0}$. Similarly, $s_0\iota_{\lambda'}\in U_{s_0}$. Thus $U_{s_0}$ intersects both $a_0^{-1}(\Zb)$ and $a_0^{-1}(1/2+\Zb)$, which contradicts  our choice of $\cU$. This proves our claim.

By \cite[Lemma 3.2]{LW} for any finite open cover $\cV$ of $[0, 1/2]^B$ with no item intersecting two opposing faces of the cube $[0, 1/2]^B$ one has $\ord(\cV)\ge |B|$. It follows
that
$$\cD(\cU^F)\ge \cD(\iota^{-1}(\cU^F))\ge |B|=\rank(\left<F^{-1}A\right>)$$
as desired.
\end{proof}

For compact spaces $X$ and $Y$ and an open cover $\cU$ of $X$, a continuous map $\varphi: X\rightarrow Y$ is said to be {\it $\cU$-compatible} if for every $y\in Y$, the set $\varphi^{-1}(y)$ is contained in some item of $\cU$.

\begin{lemma} \label{L-compatible map}
Let $\sM$ be a discrete abelian group. For any finite open cover $\cU$ of $\widehat{\sM}$, there exists an $A\in \cF(\sM)$ such that the map
$\widehat{\sM}\rightarrow (\Rb/\Zb)^A$ sending $x$ to $(x(a))_{a\in A}$ is $\cU$-compatible.
\end{lemma}
\begin{proof} By Pontryagin duality, the natural map $\widehat{\sM}\rightarrow (\Rb/\Zb)^\sM$ sending $x$ to $(x(a))_{a\in \sM}$ is an embedding. Thus we may identify $\widehat{\sM}$ with its image, and think of $\widehat{\sM}$ as a closed subset of $(\Rb/\Zb)^\sM$.
For each $x\in \widehat{\sM}$, we can find some $A_x\in \cF(\sM)$ and an open neighborhood $V_{x, a}$ of $x(a)$ in $\Rb/\Zb$ for each $a\in A_x$ such that the open subset $W_x:=\{y\in \widehat{\sM}: y(a)\in V_{x, a} \mbox{ for all } a\in A_x\}$ of $\widehat{\sM}$ is contained in some item of $\cU$. Since $\widehat{\sM}$ is compact, we can find some $Y\in \cF(\widehat{\sM})$ such that $\{W_x\}_{x\in Y}$ covers $\widehat{\sM}$. Set $A=\bigcup_{x\in Y} A_x\in \cF(\sM)$.
Then the corresponding map $\widehat{\sM}\rightarrow (\Rb/\Zb)^A$ sending $x$ to $(x(a))_{a\in A}$ is $\cU$-compatible.
\end{proof}

\begin{lemma} \label{L-mdim<mrank}
For any $\Zb\Gamma$-module $\cM$, one has $\mdim(\widehat{\cM})\le \mrank(\cM)$.
\end{lemma}
\begin{proof}
It suffices to show that for any finite open cover $\cU$ of $\widehat{\cM}$, one has $\lim_F\frac{\cD(\cU^F)}{|F|}\le \mrank(\cM)$. By Lemma~\ref{L-compatible map} we can find some $A\in \cF(\cM)$ such that the map
$\psi: \widehat{\cM}\rightarrow (\Rb/\Zb)^A$ sending $x$ to $(x(a))_{a\in A}$ is $\cU$-compatible. Then it suffices to show $\cD(\cU^F)\le \rank(\left<F^{-1}A\right>)$ for every $F\in \cF(\Gamma)$.

Denote by $\psi_F$ the map $\widehat{\cM}\rightarrow ((\Rb/\Zb)^A)^F$ sending $x$ to $(\psi(sx))_{s\in F}$.
Since $\psi$ is $\cU$-compatible, $\psi_F$ is $\cU^F$-compatible. Denote by $Z_F$ the image of $\psi_F$. By \cite[Proposition 2.4]{LW} for any compact Hausdorff space $X$ and any finite open cover $\cV$ of $X$, if there is a continuous $\cV$-compatible map from $X$ into some compact Hausdorff space $Y$, then $\cD(\cV)\le \dim(Y)$. Thus
$\cD(\cU^F)\le \dim(Z_F)$.

Note that $\psi_F$ is a group homomorphism, and hence $Z_F$ is a quotient group of $\widehat{\cM}$. By Pontryagin duality $Z_F=\widehat{\sM_F}$ for some subgroup $\sM_F$ of $\cM$. We may decompose $\varphi_F: \widehat{\cM}\rightarrow ((\Rb/\Zb)^A)^F$ naturally as $\widehat{\cM}\twoheadrightarrow Z_F\hookrightarrow ((\Rb/\Zb)^A)^F$. The corresponding dual map is
$\cM\hookleftarrow \sM_F\twoheadleftarrow (\Zb^A)^F$. The map $\cM\leftarrow (\Zb^A)^F$ sends $(\lambda_{a, s})_{a\in A, s\in F}$ to $\sum_{a\in A, s\in F}\lambda_{a, s}s^{-1}a$. Thus $\sM_F$ is equal to the image of $\cM\leftarrow (\Zb^A)^F$, which is exactly $\left<F^{-1}A\right>$.

By the result of Pontryagin \cite[page 259]{Pontryagin} we have $\dim(Z_F)=\rank(\sM_F)$. (Actually here $\sM_F$ is a finitely generated abelian group, hence it is easy to obtain $\dim(Z_F)=\rank(\sM_F)$.) Therefore
$$\cD(\cU^F)\le \dim(Z_F)=\rank(\sM_F)=\rank(\left<F^{-1}A\right>)$$
as desired.
\end{proof}

\section{Mean rank and von Neumann-L\"{u}ck rank} \label{S-mrank vs vrank}

In this section, we prove the following

\begin{theorem} \label{T-mrank vs vrank}
For any $\Zb\Gamma$-module $\cM$, one has $\mrank(\cM)=\vrank(\cM)$.
\end{theorem}

Theorem~\ref{T-main} follows from Theorems~\ref{T-mdim vs mrank} and \ref{T-mrank vs vrank}.

\subsection{Finitely presented case} \label{SS-finitely presented case}

In this subsection we prove Theorem~\ref{T-mrank vs vrank} for finitely presented $\Zb\Gamma$-modules. We need the following well-known right exactness of tensor functor \cite[Proposition 19.13]{AF} several times:

\begin{lemma} \label{L-right exact}
Let $R$ be a unital ring, and $M$ be a right $R$-module. For any exact sequence
$$M_1\rightarrow M_2\rightarrow M_3\rightarrow 0$$
of left $R$-modules, the sequence
$$M\otimes_RM_1\rightarrow M\otimes_RM_2\rightarrow M\otimes_RM_3\rightarrow 0$$
of abelian groups is exact.
\end{lemma}

Let $\cM$ be a finitely presented (left) $\Zb\Gamma$-module. Say, $\cM=(\Zb\Gamma)^{1\times n}/(\Zb\Gamma)^{1\times m}f$ for some $n, m\in \Nb$ and $f\in M_{m, n}(\Zb\Gamma)$.
Denote by $\ker f$ the kernel of the bounded linear operator $(\ell^2(\Gamma))^{n\times 1}\rightarrow (\ell^2(\Gamma))^{m\times 1}$ sending $z$ to $fz$, and by
$P_f$ the orthogonal projection from $(\ell^2(\Gamma))^{n\times 1}$ onto $\ker f$. Note that $\ker f$ is invariant under the direct sum of the right regular representation of $\Gamma$ on $(\ell^2(\Gamma))^{n\times 1}$. Thus $P_f$ commutes with this representation, and hence $P_f\in M_n(\cN\Gamma)$.

For any subset $K$ of $\Gamma$, we denote by $\Cb[K]$ the subspace of $\ell^2(\Gamma)$ and $\ell^\infty(\Gamma)$ consisting of elements vanishing on $\Gamma\setminus K$. Similarly, we have $\Rb[K]$, $\Qb[K]$ and $\Zb[K]$.

We need the following result of Elek \cite{Elek}:

\begin{lemma} \label{L-Elek}
One has
$$ \tr_{\cN\Gamma}P_f=\lim_F\frac{\dim_{\Cb}(\ker f\cap (\Cb[F])^{n\times 1})}{|F|}.$$
\end{lemma}

For any bounded linear operator $T:(\ell^2(\Gamma))^{n\times 1}\rightarrow (\ell^2(\Gamma))^{m\times 1}$ one has the {\it polar decomposition} as follows: there exist unique bounded linear operators $U: (\ell^2(\Gamma))^{n\times 1}\rightarrow (\ell^2(\Gamma))^{m\times 1})$ and $S:(\ell^2(\Gamma))^{n\times 1}\rightarrow
(\ell^2(\Gamma))^{n\times 1}$ satisfying that $\left<Sx, x\right>\ge 0$ for all $x\in (\ell^2(\Gamma))^{n\times 1}$, $\ker U=\ker S=\ker T$,  $U$ is an isometry from the orthogonal complement of $\ker T$ onto the closure of $\im T$,
and $T=US$ \cite[Theorem 6.1.2]{KR2}. When $T\in M_{m, n}(\cN\Gamma)$, since $T$ is fixed under the adjoint action of $\Gamma$ on the space of all bounded linear operators $(\ell^2(\Gamma))^{n\times 1}\rightarrow (\ell^2(\Gamma))^{m\times 1}$ via the direct sums of the right regular representation of $\Gamma$ on $(\ell^2(\Gamma))^{n\times 1}$ and $(\ell^2(\Gamma))^{m\times 1}$, both $U$ and $S$ are also fixed under the adjoint actions of $\Gamma$, and hence $U\in M_{m, n}(\cN\Gamma)$ and $S\in M_n(\cN\Gamma)$.

\begin{lemma} \label{L-ker and dim}
For any discrete (not necessarily amenable) group $\Gamma$,
one has
$$\tr_{\cN\Gamma}P_f=\vrank(\cM).$$
\end{lemma}
\begin{proof}
From the exact sequence
$$ (\Zb\Gamma)^{1\times m}\overset{\cdot f}\rightarrow (\Zb\Gamma)^{1\times n}\rightarrow \cM\rightarrow 0$$
of left $\Zb\Gamma$-modules, by Lemma~\ref{L-right exact}
we have the exact sequence
$$0\rightarrow \fM\rightarrow (\cN\Gamma)^{1\times m}\overset{\cdot f}\rightarrow (\cN\Gamma)^{1\times n}\rightarrow \cN\Gamma\otimes_{\Zb\Gamma}\cM\rightarrow 0$$
of left $\cN\Gamma$-modules, where $\fM:=\{x\in (\cN\Gamma)^{1\times m}: xf=0\}$.
By Theorem~\ref{T-dim basic}  we get
$$ \dim_{\cN\Gamma}(\cN\Gamma\otimes_{\Zb\Gamma}\cM)+m=\dim_{\cN\Gamma}\fM+n.$$

Let $f=US$ be the polar decomposition of $f$.
Denote by $I_m$ the identify matrix in $M_m(\cN\Gamma)$.
We claim that $\fM=(\cN\Gamma)^{1\times m}(I_m-UU^*)$. Since $(I_m-UU^*)U=0$, we have
$$   (\cN\Gamma)^{1\times m}(I_m-UU^*)f=(\cN\Gamma)^{1\times m}(I_m-UU^*)US=\{0\},$$
and hence $\fM\supseteq (\cN\Gamma)^{1\times m}(I_m-UU^*)$. Let $x\in \fM$. Then $xUS=xf=0$. Thus $xUS(\ell^2(\Gamma))^{n\times 1}=\{0\}$, and hence $$xU(\ell^2(\Gamma))^{n\times 1}=xUU^*U(\ell^2(\Gamma))^{n\times 1}=xU\overline{S(\ell^2(\Gamma))^{n\times 1}}=\{0\}.$$
That is, $xU=0$. Therefore $x=x(I_m-UU^*)\in (\cN\Gamma)^{1\times m}(I_m-UU^*)$. This proves our claim.

Note that $I_m-UU^*$ is the orthogonal projection from $(\ell^2(\Gamma))^{m\times 1}$ onto the orthogonal complement of $\im U$.
Thus $I_m-UU^*$ is
an idempotent in $M_m(\cN\Gamma)$. By Theorem~\ref{T-dim basic}
we have
$$ \dim_{\cN\Gamma}\fM=\tr_{\cN\Gamma}(I_m-UU^*)=m-\tr_{\cN\Gamma}(U^*U).$$
Note that $P_f=I_n-U^*U$. Thus
\begin{align*}
\vrank(\cM)&=\dim_{\cN\Gamma}(\cN\Gamma\otimes_{\Zb\Gamma}\cM)\\
&=\dim_{\cN\Gamma}\fM+n-m\\
&=n-\tr_{\cN\Gamma}(U^*U)\\
&=\tr_{\cN\Gamma}(I_n-U^*U)\\
&=\tr_{\cN\Gamma}P_f.
\end{align*}
\end{proof}

Now we prove Theorem~\ref{T-mrank vs vrank} for finitely presented $\Zb\Gamma$-modules.

\begin{lemma} \label{L-mrank vs vrank}
For any finitely presented $\Zb\Gamma$-module $\cM$, one has
$$\mrank(\cM)=\vrank(\cM).$$
\end{lemma}
\begin{proof} Say, $\cM=(\Zb\Gamma)^{1\times n}/(\Zb\Gamma)^{1\times m}f$ for some $n, m\in \Nb$ and $f\in M_{m, n}(\Zb\Gamma)$.

Let $F\in \cF(\Gamma)$. Since $f^*$ has integral coefficients, we have
$$ \rank(\ker f^*\cap (\Zb[F])^{m\times 1})=\dim_\Qb(\ker f^*\cap (\Qb[F])^{m\times 1})=\dim_\Cb(\ker f^*\cap (\Cb[F])^{m\times 1}).$$

Denote by $A$ the set of all rows of $f$. Then  $(\Zb\Gamma)^{1\times m}f$ is the $\Zb\Gamma$-submodule of $(\Zb\Gamma)^{1\times n}$ generated by $A$.
Note that we have  a short exact sequence
$$ 0\rightarrow \ker f^*\cap (\Zb[F])^{m\times 1}\rightarrow (\Zb[F])^{m\times 1}\overset{f^*\cdot}\rightarrow \left<A^*F\right>\rightarrow 0$$
of abelian groups.
Then
\begin{align*}
\mrank((\Zb\Gamma)^{1\times m}f)&\overset{Lemma~\ref{L-finitely generated}}=\lim_F\frac{\rank(\left<F^{-1}A\right>)}{|F|}\\
&=\lim_F\frac{\rank(\left<A^*F\right>)}{|F|}\\
&\overset{Lemma~\ref{L-rank}}=\lim_F\frac{\rank((\Zb[F])^{m\times 1})-\rank(\ker f^*\cap (\Zb[F])^{m\times 1})}{|F|}\\
&=\lim_F\frac{m|F|-\dim_\Cb(\ker f^*\cap (\Cb[F])^{m\times 1})}{|F|}\\
&\overset{Lemma~\ref{L-Elek}}=m-\tr_{\cN\Gamma}P_{f^*}.
\end{align*}

Let $f=US$ be the polar decomposition of $f$. Then
\begin{align*}
 m-\tr_{\cN\Gamma}P_{f^*}&=\tr_{\cN\Gamma}(I_m-P_{f^*})=\tr_{\cN\Gamma}(UU^*)=\tr_{\cN\Gamma}(U^*U)\\
 &=\tr_{\cN\Gamma}(I_n-P_f)=n-\tr_{\cN\Gamma}P_f,
\end{align*}
where the third equality follows from the tracial property \eqref{E-trace} of $\tr_{\cN\Gamma}$.

Taking $\cM_2=(\Zb\Gamma)^{1\times n}$ and $\cM_1=(\Zb\Gamma)^{1\times m}f$ in  Lemma~\ref{L-addition2} we get
\begin{align*}
\mrank(\cM)&=\mrank((\Zb\Gamma)^{1\times n})-\mrank((\Zb\Gamma)^{1\times m}f)\\
&=n-(m-\tr_{\cN\Gamma}P_{f^*})\\
&=n-(n-\tr_{\cN\Gamma}P_f)\\
&=\tr_{\cN\Gamma}P_f\\
&\overset{Lemma~\ref{L-ker and dim}}=\vrank(\cM).
\end{align*}
\end{proof}

\subsection{General case} \label{SS-general case}

The next lemma is the analogue of Lemma~\ref{L-decreasing limit mrank} for von Neumann-L\"{u}ck rank.
Lemmas~\ref{L-decreasing limit mrank} and \ref{L-decreasing limit vrank} together will enable us to pass from finitely presented modules to finitely generated modules in the proof of Theorem~\ref{T-mrank vs vrank}.

\begin{lemma} \label{L-decreasing limit vrank}
Let $\Gamma$ be a discrete (not necessarily amenable) group. Let $\cM$ be a finitely generated $\Zb\Gamma$-module. Let $\{\cM_n\}_{n\in J}$ be an increasing net of submodules of $\cM$ with $\infty\not\in J$. Set $\cM_\infty=\bigcup_{n\in J}\cM_n$. Then
$$\vrank(\cM/\cM_\infty)=\lim_{n\to\infty}\vrank(\cM/\cM_n)=\inf_{n\in J} \vrank(\cM/\cM_n).$$
\end{lemma}
\begin{proof} For any $n<m$ in $J$, we have the exact sequence
$$ 0\rightarrow \cM_m/\cM_n\rightarrow \cM/\cM_n\rightarrow \cM/\cM_m\rightarrow 0$$
of $\Zb\Gamma$-modules. By Lemma~\ref{L-right exact}
we obtain  the exact sequence
$$ \cN\Gamma\otimes_{\Zb\Gamma}(\cM_m/\cM_n)\rightarrow \cN\Gamma\otimes_{\Zb\Gamma}(\cM/\cM_n)\rightarrow \cN\Gamma\otimes_{\Zb\Gamma}(\cM/\cM_m)\rightarrow 0$$
of $\cN\Gamma$-modules. By Theorem~\ref{T-dim basic},
we have
$$\vrank(\cM/\cM_n)=\dim_{\cN\Gamma}( \cN\Gamma\otimes_{\Zb\Gamma}(\cM/\cM_n))\ge \dim_{\cN\Gamma}( \cN\Gamma\otimes_{\Zb\Gamma}(\cM/\cM_m))=\vrank(\cM/\cM_m).$$
Thus
$$\lim_{n\to\infty}\vrank(\cM/\cM_n)=\inf_{n\in J} \vrank(\cM/\cM_n).$$

For each $n\in J\cup \{\infty\}$, denote by $\pi_n$ the surjective homomorphism $ \cN\Gamma\otimes_{\Zb\Gamma}\cM\rightarrow \cN\Gamma\otimes_{\Zb\Gamma}(\cM/\cM_n)$.
From  the exact sequence
$$0\rightarrow \cM_n\rightarrow \cM\rightarrow \cM/\cM_n\rightarrow 0$$
of $\Zb\Gamma$-modules, by Lemma~\ref{L-right exact}
 we obtain  the exact sequence
$$ \cN\Gamma\otimes_{\Zb\Gamma}\cM_n\rightarrow \cN\Gamma\otimes_{\Zb\Gamma}\cM\rightarrow \cN\Gamma\otimes_{\Zb\Gamma}(\cM/\cM_n)\rightarrow 0$$
of $\cN\Gamma$-modules for each $n\in J\cup \{\infty\}$. Thus $\ker \pi_n$ is equal to the image of the homomorphism $\cN\Gamma\otimes_{\Zb\Gamma}\cM_n\rightarrow \cN\Gamma\otimes_{\Zb\Gamma}\cM$. It follows that $\ker \pi_\infty=\bigcup_{n\in J}\ker \pi_n$. By Theorem~\ref{T-dim basic}, we have
\begin{align*} \dim_{\cN\Gamma}(\cN\Gamma\otimes_{\Zb\Gamma}\cM)&=\dim_{\cN\Gamma}(\ker \pi_n)+\dim_{\cN\Gamma}(\cN\Gamma\otimes_{\Zb\Gamma}(\cM/\cM_n))\\
&=\dim_{\cN\Gamma}(\ker \pi_n)+\vrank(\cM/\cM_n)
\end{align*}
for all $n\in J\cup \{\infty\}$,
and
$$ \dim_{\cN\Gamma}(\ker \pi_\infty)=\lim_{n\to \infty}\dim_{\cN\Gamma}(\ker \pi_n).$$
It follows that
$$\dim_{\cN\Gamma}(\ker \pi_\infty)+\vrank(\cM/\cM_\infty)=\dim_{\cN\Gamma}(\ker \pi_\infty)+\lim_{n\to \infty}\vrank(\cM/\cM_n).$$
Since $\cM$ is a finitely generated $\Zb\Gamma$-module,
$\cN\Gamma\otimes_{\Zb\Gamma}\cM$ is a finitely generated $\cN\Gamma$-module. By Theorem~\ref{T-dim basic} we have
$\dim_{\cN\Gamma}(\ker \pi_\infty)\le \dim_{\cN\Gamma}(\cN\Gamma\otimes_{\Zb\Gamma}\cM)<+\infty$. Therefore
$$\vrank(\cM/\cM_\infty)=\lim_{n\to \infty}\vrank(\cM/\cM_n).$$
\end{proof}

The following lemma is the analogue of Theorem~\ref{T-addition mrank} for von Neumann-L\"{u}ck rank.

\begin{lemma} \label{L-addition vrank}
For any short exact sequence
$$0\rightarrow \cM_1\rightarrow \cM_2\rightarrow \cM_3\rightarrow 0$$
of $\Zb\Gamma$-modules, one has
$$ \vrank(\cM_2)=\vrank(\cM_1)+\vrank(\cM_3).$$
\end{lemma}
\begin{proof} Note that
$ \Cb\otimes_\Zb\Zb\Gamma=\Cb\Gamma$
and hence for any  $\Zb\Gamma$-module $\cM$ one has
$$\Cb\Gamma\otimes_{\Zb\Gamma}\cM=(\Cb\otimes_\Zb\Zb\Gamma)\otimes_{\Zb\Gamma}\cM=\Cb\otimes_\Zb(\Zb\Gamma\otimes_{\Zb\Gamma}\cM)=\Cb\otimes_\Zb\cM.$$
Since $\Cb$ is a torsion-free $\Zb$-module, the functor $\Cb\otimes_\Zb\cdot$ from the category of $\Zb$-modules to the category of $\Cb$-modules is exact \cite[Proposition XVI.3.2]{Lang}. Thus the functor $\Cb\Gamma\otimes_{\Zb\Gamma}\cdot$ from the category of (left) $\Zb\Gamma$-modules to the category of (left) $\Cb\Gamma$-modules is exact. Therefore
we have the short exact sequence
\begin{align} \label{E-exact}
 0\rightarrow \Cb\Gamma\otimes_{\Zb\Gamma}\cM_1\rightarrow \Cb\Gamma\otimes_{\Zb\Gamma}\cM_2\rightarrow \Cb\Gamma\otimes_{\Zb\Gamma}\cM_3\rightarrow 0
\end{align}
of $\Cb\Gamma$-modules. Note that
$$\cN\Gamma\otimes_{\Cb\Gamma}(\Cb\Gamma\otimes_{\Zb\Gamma}\cM)=(\cN\Gamma\otimes_{\Cb\Gamma}\Cb\Gamma)\otimes_{\Zb\Gamma}\cM=\cN\Gamma\otimes_{\Zb\Gamma}\cM$$
for every $\Zb\Gamma$-module $\cM$. By Lemma~\ref{L-right exact}
we have an exact sequence
$$ 0\rightarrow \fM\rightarrow \cN\Gamma\otimes_{\Zb\Gamma}\cM_1\rightarrow \cN\Gamma\otimes_{\Zb\Gamma}\cM_2\rightarrow \cN\Gamma\otimes_{\Zb\Gamma}\cM_3\rightarrow 0$$
of $\cN\Gamma$-modules, where $\fM$ is the kernel of the homomorphism $\cN\Gamma\otimes_{\Zb\Gamma}\cM_1\rightarrow \cN\Gamma\otimes_{\Zb\Gamma}\cM_2$.
L\"{u}ck showed that $\cN\Gamma$ as a right $\Cb\Gamma$-module is dimension flat, i.e. for any $\Cb\Gamma$-chain complex $\cC_*$ with $\cC_p=0$ for all $p<0$, one has \cite[page 275]{Luck}
$$ \dim_{\cN\Gamma}(H_p(\cN\Gamma\otimes_{\Cb\Gamma}\cC_*))=\dim_{\cN\Gamma}(\cN\Gamma\otimes_{\Cb\Gamma}H_p(\cC_*))$$
for all $p\in \Zb$.
Taking $\cC_*$ to be the exact sequence \eqref{E-exact}, we get
$$\dim_{\cN\Gamma}\fM=\dim_{\cN\Gamma}(H_2(\cN\Gamma\otimes_{\Cb\Gamma}\cC_*))=\dim_{\cN\Gamma}(\cN\Gamma\otimes_{\Cb\Gamma}H_2(\cC_*))
=\dim_{\cN\Gamma}(\cN\Gamma\otimes_{\Cb\Gamma}0)=0.$$
By Theorem~\ref{T-dim basic} we have
\begin{align*}
\vrank(\cM_1)+\vrank(\cM_3)&=\dim_{\cN\Gamma}(\cN\Gamma\otimes_{\Zb\Gamma}\cM_1)+\dim_{\cN\Gamma}(\cN\Gamma\otimes_{\Zb\Gamma}\cM_3)\\
&=\dim_{\cN\Gamma}\fM+\dim_{\cN\Gamma}(\cN\Gamma\otimes_{\Zb\Gamma}\cM_2)\\
&=\dim_{\cN\Gamma}(\cN\Gamma\otimes_{\Zb\Gamma}\cM_2)=\vrank(\cM_2).
\end{align*}
\end{proof}

The next lemma is the analogue of Lemma~\ref{L-increasing limit mrank} for von Neumann-L\"{u}ck rank.
Lemmas~\ref{L-increasing limit mrank} and \ref{L-increasing limit vrank} together
will enable us to pass from finitely generated modules to arbitrary modules in the proof of Theorem~\ref{T-mrank vs vrank}.

\begin{lemma} \label{L-increasing limit vrank}
Let $\cM$ be a $\Zb\Gamma$-module and $\{\cM_n\}_{n\in J}$ be an increasing net of submodules of $\cM$ with union $\cM$. Then $$\vrank(\cM)=
\lim_{n\to\infty}\vrank(\cM_n)=\sup_{n\in J}\vrank(\cM_n).$$
\end{lemma}
\begin{proof} From Lemma~\ref{L-addition vrank} we have $\lim_{n\to\infty}\vrank(\cM_n)=\sup_{n\in J}\vrank(\cM_n).$

For each $n\in J$ denote by $\fM_n$ the image of the natural homomorphism $\cN\Gamma\otimes_{\Zb\Gamma}\cM_n\rightarrow \cN\Gamma\otimes_{\Zb\Gamma}\cM$. Then $\{\fM_n\}_{n\in J}$ is an increasing net of submodules of $\cN\Gamma\otimes_{\Zb\Gamma}\cM$ with union $\cN\Gamma\otimes_{\Zb\Gamma}\cM$. By Theorem~\ref{T-dim basic} we have
$$ \vrank(\cM)=\dim_{\cN\Gamma}(\cN\Gamma\otimes_{\Zb\Gamma}\cM)=\sup_{n\in J}\dim_{\cN\Gamma}\fM_n.$$

Let $n\in J$. Taking $\cM_1=\cM_n$ and $\cM_2=\cM$ in Lemma~\ref{L-addition vrank}, the argument in the proof there shows that $\dim_{\cN\Gamma}\fM'_n=0$ for
$\fM'_n$ denoting
the kernel of the homomorphism $\cN\Gamma\otimes_{\Zb\Gamma}\cM_n\rightarrow \cN\Gamma\otimes_{\Zb\Gamma}\cM$. From Theorem~\ref{T-dim basic} we then conclude that $\dim_{\cN\Gamma}\fM_n=\dim_{\cN\Gamma}(\cN\Gamma\otimes_{\Zb\Gamma}\cM_n)=\vrank(\cM_n)$. Therefore
$$ \vrank(\cM)=\sup_{n\in J}\dim_{\cN\Gamma}\fM_n=\sup_{n\in J}\vrank(\cM_n).$$
\end{proof}

We are ready to prove Theorem~\ref{T-mrank vs vrank}.

\begin{proof}[Proof of Theorem~\ref{T-mrank vs vrank}] By Lemma~\ref{L-mrank vs vrank} we know that the theorem holds for all finitely presented $\Zb\Gamma$-modules.

Let $\cM$ be a finitely generated $\Zb\Gamma$-module. Then $\cM=(\Zb\Gamma)^m/\cM_\infty$ for some $m\in \Nb$ and some submodule $\cM_\infty$ of $(\Zb\Gamma)^m$.
Write $\cM_\infty$ as the union of an increasing net of finitely generated submodules $\{\cM_n\}_{n\in J}$ with $\infty\not \in J$.  Then $(\Zb\Gamma)^m/\cM_n$ is a finitely presented $\Zb\Gamma$-module for each $n\in J$, and hence
\begin{align*}
\mrank(\cM)&\overset{Lemma~\ref{L-decreasing limit mrank}}=\lim_{n\to \infty}\mrank((\Zb\Gamma)^m/\cM_n)\\
&=\lim_{n\to \infty} \vrank((\Zb\Gamma)^m/\cM_n)\\
&\overset{Lemma~\ref{L-decreasing limit vrank}}=\vrank(\cM).
\end{align*}

Finally, since every $\Zb\Gamma$-module is the union  of an increasing net of finitely generated submodules, by Lemmas~\ref{L-increasing limit mrank} and \ref{L-increasing limit vrank} we conclude that the theorem holds for every $\Zb\Gamma$-module.
\end{proof}

\section{Applications} \label{S-applications}

Every compact metrizable space is a quotient space of the Cantor set. Thus dimension does not necessarily decrease when passing to quotient spaces of compact metrizable spaces. Similarly, one can show that mean dimension does not necessarily decrease when passing to factors of continuous $\Gamma$-actions on compact metrizable spaces, where we say that a continuous $\Gamma$-action on a compact metrizable space $Y$ is a {\it factor} of a continuous $\Gamma$-action on a compact metrizable space $X$ if there is a surjective continuous $\Gamma$-equivariant map $X\rightarrow Y$.
On the other hand, in algebraic setting, dimension does decrease when passing to quotient space:
for any compact metrizable group $X$ and any closed subgroup $Y$, one has $\dim(X)=\dim(Y)+\dim(X/Y)$ \cite{Yamanoshita} \cite[Theorem 3]{Karube}.
From Theorems~\ref{T-addition mrank} and \ref{T-mdim vs mrank}, we obtain that, in algebraic setting, mean dimension also decreases when passing to factors in algebraic situation:

\begin{corollary} \label{C-addition mdim}
For any short exact sequence
$$0\rightarrow X_1\rightarrow X_2\rightarrow X_3\rightarrow 0$$
of compact metrizable abelian groups carrying continuous $\Gamma$-actions by automorphisms and  $\Gamma$-equivariant continuous homomorphisms, we have
$$ \mdim(X_2)=\mdim(X_1)+\mdim(X_3).$$
In particular, $\mdim(X_2)\ge \mdim(X_3)$.
\end{corollary}

The analogue of the addition formula for entropy was established in \cite[Corollary 6.3]{Li}, and is crucial for the proof of the relation between entropy and $L^2$-torsion in \cite{LT}.

The following is an application to von Neumann-L\"{u}ck rank.

\begin{corollary} \label{C-vrank}
Suppose that $\Gamma$ is infinite and $\cM$ is a $\Zb\Gamma$-module with finite rank as an abelian group. Then $\vrank(\cM)=0$.
\end{corollary}
\begin{proof}
For any $A\in \cF(\cM)$ and $F\in \cF(\Gamma)$, one has $\rank(\left< F^{-1}A\right>)\le \rank(\cM)<+\infty$, and hence
$\lim_F\frac{\rank(\left<F^{-1}A\right>)}{|F|}=0$. Thus by Theorem~\ref{T-mrank vs vrank} we have $\vrank(\cM)=\mrank(\cM)=0$.
\end{proof}

\begin{example} \label{Ex-vrank}
Let $\varphi$ be a group homomorphism from $\Gamma$ into $\GL_n(\Qb)$ for some $n\in \Nb$, where $\GL_n(\Qb)$ denotes the group of all invertible elements in
$M_n(\Qb)$. Then $\Gamma$ acts on $\Qb^{n\times 1}$ via $\varphi$, and hence $\Qb^{n\times 1}$ becomes a $\Zb\Gamma$-module. Note that the rank of $\Qb^{n\times 1}$ as abelian group is equal to $n$. Thus when $\Gamma$ is infinite, by Corollary~\ref{C-vrank} we have $\vrank(\Qb^{n\times 1})=0$.
\end{example}

\section{Metric mean dimension} \label{S-metric mdim}

Throughout this section $\Gamma$ will be a discrete countable amenable group.

Lindenstrauss and Weiss also introduced a dynamical analogue of the Minkowski dimension. Let $\Gamma$ act on a compact metrizable space $X$ continuously, and $\rho$ be a continuous pseudometric on $X$. For $\varepsilon>0$, a subset $Y$ of $X$ is called {\it $(\rho, \varepsilon)$-separated} if $\rho(x, y)\ge \varepsilon$ for any distinct $x$ and $y$ in $Y$. Denote by $N_\varepsilon(X, \rho)$ the maximal cardinality of $(\rho, \varepsilon)$-separated subsets of $X$. For any $F\in \cF(\Gamma)$, we define a new continuous pseudometric $\rho_F$ on $X$ by $\rho_F(x, y)=\max_{s\in F}\rho(sx, sy)$. The {\it metric mean dimension} of the action $\Gamma\curvearrowright X$ with respect to $\rho$ \cite[Definition 4.1]{LW}, denoted by $\mdim_\rM(X, \rho)$, is defined as
$$ \mdim_\rM(X, \rho):=\varliminf_{\varepsilon\to 0}\frac{1}{|\log \varepsilon|}\varlimsup_F\frac{\log N_\varepsilon(X, \rho_F)}{|F|}.$$
When $\Gamma$ is the trivial group, the metric mean dimension is exactly the Minkowski dimension of $(X, \rho)$.

A well-known theorem of Pontryagin and Schnirelmann \cite{PS} \cite[page 80]{Nagata}
says that for any compact metrizable space $X$, the  covering dimension of $X$ is equal to the minimal value of the Minkowski dimension of $(X, \rho)$ for $\rho$ ranging over compatible metrics on $X$.
Since mean dimension and metric mean dimension
are dynamical analogues of  covering dimension and Minkowski dimension respectively,
it is natural to ask

\begin{question} \label{Q-top vs metric}
Let $\Gamma$ act continuously on a compact metrizable space $X$. Is $\mdim(X)$ equal to the minimal value of $\mdim_\rM(X, \rho)$ for $\rho$ ranging over compatible metrics on $X$?
\end{question}

Lindenstrauss and Weiss showed that $\mdim(X)\le \mdim_\rM(X, \rho)$ for every compatible metric $\rho$ on $X$ \cite[Theorem 4.2]{LW}. Thus Question~\ref{Q-top vs metric} reduces to the question whether $\mdim(X)=\mdim_\rM(X, \rho)$ for some compatible metric $\rho$ on $X$. For $\Gamma=\Zb$, Lindenstrauss showed that this is true when the action $\Zb\curvearrowright X$ has an infinite minimal factor \cite[Theorem 4.3]{Lindenstrauss}, and Gutman showed that this is true when
$\Zb\curvearrowright X$ has the so-called {\it marker property} \cite[Definition 3.1, Theorem 8.1]{Gutman13}, in particular when $\Zb\curvearrowright X$ has an aperiodic factor $\Zb\curvearrowright Y$ such that either $Y$ is finite-dimensional or $\Zb\curvearrowright Y$ has only a countable number of closed invariant minimal subsets or $\Zb\curvearrowright Y$ has a so-called {\it compact minimal subsystems selector}
\cite[Theorem 8.2, Definition 3.8]{Gutman13}.

We answer Question~\ref{Q-top vs metric} affirmatively for algebraic actions. In fact, we prove a stronger conclusion that one can even find a translation-invariant metric with minimal metric mean dimension. Recall that a pseudometric $\rho$ on a compact abelian group is said to be {\it translation-invariant} if $\rho(x+y, x+z)=\rho(y, z)$ for all $x, y, z\in X$.

\begin{theorem} \label{T-metric mdim}
Let $\Gamma$ act on a compact metrizable abelian group $X$ by continuous automorphisms. Then
there is a  translation-invariant compatible metric $\rho$ on $X$ satisfying $\mdim(X)=\mdim_\rM(X, \rho)$.
\end{theorem}

For a continuous action $\Gamma\curvearrowright X$ of $\Gamma$ on some compact metrizable space $X$, we say that a continuous pseudometric $\rho$ on $X$ is {\it dynamically generating} if $\sup_{s\in \Gamma}\rho(sx, sy)>0$ for any distinct $x, y\in X$. It is well known that for any dynamically generating continuous pseudometric $\rho$ on $X$ one can turn it into a compatible metric without changing the metric mean dimension, see for example \cite[page 238]{Lindenstrauss}.
For completeness, we give a proof here.

\begin{lemma} \label{L-pseudometric to metric}
Let $\Gamma$ act continuously on a compact metrizable space $X$, and $\rho$ be a dynamically-generating continuous pseudometric on $X$.
List the elements of $\Gamma$ as $s_1, s_2, \cdots$. For $x, y\in X$ set $\tilde{\rho}(x, y)=\max_{j\in \Nb}2^{-j}\rho(s_jx, s_jy)$. Then $\tilde{\rho}$ is a compatible metric
on $X$, and $\mdim_\rM(X, \rho)=\mdim_\rM(X, \tilde{\rho})$. In particular, one has $\mdim(X)\le \mdim_\rM(X, \rho)$.
\end{lemma}
\begin{proof} Since $\tilde{\rho}$ is a continuous pseudometric on the compact space $X$ separating the points of $X$, the identity map from $X$ equipped the original topology to $X$ equipped with the topology induced by $\tilde{\rho}$ is continuous and hence must be a homeomorphism. Therefore $\tilde{\rho}$ is a compatible metric on $X$.

Say, $e=s_k$.
Note that $\mdim_\rM(X, \rho)=\mdim_\rM(X, 2^{-k}\rho)\le \mdim_\rM(X, \tilde{\rho})$.

Let $\varepsilon>0$. Take $m\in \Nb$ such that $2^{-m}\diam(X, \rho)\le \varepsilon$. Set $K=\{s_1,\dots, s_m\}\in \cF(\Gamma)$. For any $x, y\in X$ and $t\in \Gamma$, one has
$$ \tilde{\rho}(tx, ty)\le \max(2^{-m-1}\diam(X, \rho), \max_{1\le j\le m}2^{-j}\rho(s_jtx, s_jty))\le \max(\varepsilon/2, \max_{s\in K}\rho(stx, sty)),$$
and hence for any $F\in \cF(\Gamma)$,
$$ \tilde{\rho}_F(x, y)\le \max(\varepsilon/2, \rho_{KF}(x, y)).$$
It follows that $N_\varepsilon(X, \tilde{\rho}_F)\le N_\varepsilon(X, \rho_{KF})$.   Thus
$$\varlimsup_F\frac{\log N_\varepsilon(X, \tilde{\rho}_F)}{|F|}\le \varlimsup_F\frac{\log N_\varepsilon(X, \rho_{KF})}{|F|}= \varlimsup_F\frac{\log N_\varepsilon(X, \rho_{KF})}{|KF|}\le  \varlimsup_F\frac{\log N_\varepsilon(X, \rho_F)}{|F|}.$$
Therefore
\begin{align*}
\mdim_\rM(X, \tilde{\rho})&=\varliminf_{\varepsilon\to 0}\frac{1}{|\log \varepsilon|}\varlimsup_F\frac{\log N_\varepsilon(X, \tilde{\rho}_F)}{|F|}\\
&\le \varliminf_{\varepsilon\to 0}\frac{1}{|\log \varepsilon|}\varlimsup_F\frac{\log N_\varepsilon(X, \rho_F)}{|F|}
=\mdim_\rM(X, \rho).
\end{align*}
Consequently, $\mdim_\rM(X, \rho)=\mdim_\rM(X, \tilde{\rho})$.

Since $\mdim(X)\le \mdim_\rM(X, \tilde{\rho})$, we get $\mdim(X)\le \mdim_\rM(X, \rho)$.
\end{proof}

We prove Theorem~\ref{T-metric mdim} first for finitely presented modules, in Lemma~\ref{L-finitely presented case metric mdim}.
We need the following well-known fact, which can be proved by a simple volume comparison argument (see for example the proof of \cite[Lemma 4.10]{Pisier}).

\begin{lemma} \label{L-volume to covering}
Let $V$ be a finite-dimensional normed space over $\Rb$. Let $\varepsilon>0$. Then any $\varepsilon$-separated subset of the unit ball of $V$ has cardinality at most
$(1+\frac{2}{\varepsilon})^{\dim_\Rb (V)}$.
\end{lemma}

Consider the following metric $\vartheta$ on $\Rb/\Zb$:
\begin{align} \label{E-metric circle}
\vartheta(x+\Zb, y+\Zb):=\min_{z\in \Zb}|x-y-z|.
\end{align}
For any   $\Zb\Gamma$-module $\cM$ and any $A\in \cF(\cM)$, we define a continuous pseudometric $\vartheta^A$ on $\widehat{\cM}$ by
\begin{align} \label{E-metric}
\vartheta^A(x, y):=\max_{a\in A}\vartheta(x(a), y(a)).
\end{align}
Note that $\vartheta^A$ is dynamically generating iff $\Gamma A$ separates the points of $\cM$, iff $\left<\Gamma A\right>=\cM$, i.e. $A$ generates $\cM$
as a $\Zb\Gamma$-module.

To prove the following lemma, we use a modification of the argument in the proof of \cite[Theorem 4.11]{CL}.

\begin{lemma} \label{L-finitely presented case metric mdim}
Let $f\in M_{m, n}(\Zb\Gamma)$ for some $m, n\in \Nb$, and $\cM=(\Zb\Gamma)^{1\times n}/(\Zb\Gamma)^{1\times m}f$.
Denote by $A$ the image of the canonical generators of $(\Zb\Gamma)^{1\times n}$ under the natural homomorphism $(\Zb\Gamma)^{1\times n}\rightarrow \cM$.
Then
$$\mdim(\widehat{\cM})=\mdim_\rM(\widehat{\cM}, \vartheta^A).$$
\end{lemma}
\begin{proof} Let $\ker f$ and $P_f$ be as at the beginning of Section~\ref{SS-finitely presented case}.
By Theorem~\ref{T-main} and Lemma~\ref{L-ker and dim} one has
$\tr_{\cN\Gamma}P_f=\mdim(\widehat{\cM})$. Since by Lemma~\ref{L-pseudometric to metric} $\mdim(\widehat{\cM})\le \mdim_\rM(\widehat{\cM}, \vartheta^A)$, it suffices to show  $\mdim_\rM(\widehat{\cM}, \vartheta^A)\le \tr_{\cN\Gamma}P_f$.

We identify $\widehat{(\Zb\Gamma)^{1\times n}}$ with $((\Rb/\Zb)^{n\times 1})^\Gamma=((\Rb/\Zb)^\Gamma)^{n\times 1}$ naturally via the pairing
$(\Zb\Gamma)^{1\times n}\times ((\Rb/\Zb)^\Gamma)^{n\times 1}\rightarrow \Rb/\Zb$ given by
$$\left<g, x\right>=(gx)_e,$$
where $gx\in (\Rb/\Zb)^\Gamma$ is defined similar to the product in $\Zb\Gamma$:
$$(gx)_t=\sum_{1\le j\le n}\sum_{s\in \Gamma}g_{j, s}x_{j, s^{-1}t}.$$
Via the quotient homomorphism $(\Zb\Gamma)^{1\times n}\rightarrow \cM$, we shall identify $\widehat{\cM}$ with a closed subset of $\widehat{(\Zb\Gamma)^{1\times n}}=((\Rb/\Zb)^{n\times 1})^\Gamma$.
It is easily checked that
$$\widehat{\cM}=\{x\in ((\Rb/\Zb)^{n\times 1})^\Gamma: fx=0 \mbox{ in } ((\Rb/\Zb)^{m\times 1})^\Gamma\},$$
and the action $\sigma$ of $\Gamma$ on $\widehat{\cM}$ is given by right shift:
$$(\sigma_s(x))_{j, t}=x_{j, ts}.$$
Furthermore, for any $x, y\in \widehat{\cM}$,
$$\vartheta^A(x, y)=\max_{1\le j\le n}\vartheta(x_{j, e}, y_{j, e}),$$
and hence for any $F\in \cF(\Gamma)$,
$$ \vartheta^A_F(x, y)=\max_{s\in F}\max_{1\le j\le n}\vartheta(x_{j, s}, y_{j, s}).$$

Denote by $K$  the support of $f$ as an $M_{m, n}(\Zb)$-valued function on $\Gamma$.

Let $0<\varepsilon<1$. Let $F\in \cF(\Gamma)$. Set $F'=\{s\in F: K^{-1}s\subseteq F\}$. Take a $(\vartheta^A_F, \varepsilon)$-separated subset $\cW$ of $\widehat{\cM}$ with
$$|\cW|=N_\varepsilon(\widehat{\cM}, \vartheta^A_F).$$
For each $x\in \cW$, take $\tilde{x}\in ([-1/2, 1/2)^{n\times 1})^\Gamma$ such that $x_s=\tilde{x}_s+\Zb^{n\times 1}$ for every $s\in \Gamma$. Then $f\tilde{x}\in (\Zb^{m\times 1})^\Gamma$ for every $x\in \cW$.

Set
$$\|f\|_1=\sum_{1\le j\le m}\sum_{1\le k\le n}\sum_{s\in \Gamma}|f_{j, k, s}|.$$
For any $y'\in (\ell^\infty_\Rb(\Gamma))^{n\times 1}$, set
$$ \|y'\|_\infty=\max_{1\le j\le n}\sup_{s\in \Gamma}|y'_{j, s}|.$$
Then $\|fy'\|_\infty\le \|f\|_1 \|y'\|_\infty$ for every $y'\in (\ell^\infty_\Rb(\Gamma))^{n\times 1}$.

Denote by $p_F$ the restriction map $(\ell^\infty_{\Rb}(\Gamma))^{n\times 1}\rightarrow (\Rb[F])^{n\times 1}$. Let $x\in \cW$. Set $x'=p_F(\tilde{x})\in ([-1/2, 1/2)^{n\times 1})^F$. Note that $fx'=f\tilde{x}$ on $F'$, and $\|fx'\|_\infty\le \|f\|_1\|x'\|_\infty\le \|f\|_1/2$.
Thus $fx'$ takes values in $(\Zb\cap [-\|f\|_1/2, \|f\|_1/2])^{m\times 1}$ on $F'$. Therefore we can find some $\cW_1\subseteq \cW$ and $y\in \cW_1$ such that
$$|\cW|\le |\cW_1|(\|f\|_1+1)^{m|F'|}$$
and $fx'=fy'$ on $F'$ for all $x\in \cW_1$.

Denote by $V$ the linear subspace of $(\Rb[F])^{n\times 1}$ consisting of all $v$ satisfying $fv=0$ on $F'$. Then
$fV\subseteq (\Rb[KF\setminus F'])^{m\times 1}$, and hence
$$\dim_\Rb(V)\le \dim_\Rb(fV)+\dim_\Rb(\ker f \cap V)\le m|KF\setminus F'|+\dim_\Rb(\ker f\cap (\Rb[F])^{n\times 1}).$$
The set $\{x'-y': x\in \cW_1\}$ is contained in the unit ball of $V$ under the supremum norm. For any
distinct $x, z\in \cW_1$, one has
$$ \|(x'-y')-(z'-y')\|_\infty=\|x'-z'\|_\infty\ge \vartheta^A_F(x, z)\ge \varepsilon.$$
By Lemma~\ref{L-volume to covering} we have
$$ |\cW_1|\le (1+\frac{2}{\varepsilon})^{\dim_\Rb(V)}.$$
Therefore
$$ N_\varepsilon(\widehat{\cM}, \vartheta^A_F)=|\cW|\le (1+\frac{2}{\varepsilon})^{m|KF\setminus F'|+\dim_\Rb(\ker f\cap (\Rb[F])^{n\times 1})}(\|f\|_1+1)^{m|F'|}.$$

Since $f$ has real coefficients, for any $x, y\in (\ell^2_\Rb(\Gamma))^{n\times 1}$, one has $x+yi\in \ker f$ if and only if $x, y\in \ker f$. Thus
$ \dim_{\Cb}(\ker f\cap (\Cb[F])^{n\times 1})=\dim_\Rb(\ker f\cap (\Rb[F])^{n\times 1})$ for every $F\in \cF(\Gamma)$, and hence from Lemma~\ref{L-Elek} we get
\begin{align} \label{E-Elek}
\tr_{\cN\Gamma}P_f=\lim_F\frac{\dim_{\Rb}(\ker f\cap (\Rb[F])^{n\times 1})}{|F|}.
\end{align}

Now we have
\begin{eqnarray*}
& &\varlimsup_F\frac{\log N_\varepsilon(\widehat{\cM}, \vartheta^A_F)}{|F|}\\
&\le& \varlimsup_F\frac{(m|KF\setminus F'|+\dim_\Rb(\ker f\cap (\Rb[F])^{n\times 1}))\log (1+2\varepsilon^{-1})+m|F'|\log (\|f\|_1+1)}{|F|}\\
&\overset{\eqref{E-Elek}}=&\tr_{\cN\Gamma}P_f \cdot \log (1+2\varepsilon^{-1})+m\log (\|f\|_1+1).
\end{eqnarray*}
Therefore
\begin{align*}
\mdim_\rM(\widehat{\cM}, \vartheta)&=\varliminf_{\varepsilon\to 0}\frac{1}{|\log \varepsilon|}\varlimsup_F\frac{\log N_\varepsilon(\widehat{\cM}, \vartheta^A_F)}{|F|}\\
&\le \varliminf_{\varepsilon\to 0}\frac{\tr_{\cN\Gamma}P_f\cdot \log (1+2\varepsilon^{-1})+m\log (\|f\|_1+1)}{|\log \varepsilon|}=\tr_{\cN\Gamma}P_f.
\end{align*}
\end{proof}

Next we prove Theorem~\ref{T-metric mdim} for finitely generated modules.

\begin{lemma} \label{L-finitely generated case metric mdim}
Let $\cM$ be a finitely generated $\Zb\Gamma$-module, and $A$ be a finite generating subset of $\cM$.
Then
$$\mdim(\widehat{\cM})=\mdim_\rM(\widehat{\cM}, \vartheta^A).$$
\end{lemma}
\begin{proof} Using $A$ we may write $\cM$ as $(\Zb\Gamma)^n/\cM_\infty$ for $n=|A|$ and some submodule  $\cM_\infty$ of $(\Zb\Gamma)^n$ such that $A$ is the image of the canonical generators of $(\Zb\Gamma)^n$ under the quotient homomorphism $(\Zb\Gamma)^n\rightarrow \cM$.

Let $\{\cM_j\}_{j\in \Nb}$ be an increasing sequence of finitely generated submodules of $\cM_\infty$ with union $\cM_\infty$. By Theorem~\ref{T-mdim vs mrank}  and Lemma~\ref{L-decreasing limit mrank}  we have
\begin{align*}
\mdim(\widehat{\cM})
=\mrank(\cM)
=\lim_{j\to \infty}\mrank((\Zb\Gamma)^n/\cM_j)
=\lim_{j\to \infty}\mdim(\widehat{(\Zb\Gamma)^n/\cM_j}).
\end{align*}
For each $j\in \Nb$, denote by $A_j$ the image of the  canonical generators of $(\Zb\Gamma)^n$ under the quotient homomorphism $(\Zb\Gamma)^n\rightarrow (\Zb\Gamma)^n/\cM_j$.
By Lemma~\ref{L-finitely presented case metric mdim},
one has
\begin{align*}
\mdim(\widehat{\cM})
=\lim_{j\to \infty}\mdim(\widehat{(\Zb\Gamma)^n/\cM_j})=\lim_{j\to \infty}\mdim_\rM(\widehat{(\Zb\Gamma)^n/\cM_j}, \vartheta^{A_j}).
\end{align*}
From the quotient homomorphism $(\Zb\Gamma)^n/\cM_j\rightarrow (\Zb\Gamma)^n/\cM_\infty=\cM$, we may identify $\widehat{\cM}$ with a closed subgroup of $\widehat{(\Zb\Gamma)^n/\cM_j}$. Note that $\vartheta^{A_j}$ restricts to $\vartheta^A$ on $\widehat{\cM}$. Thus
 $\mdim_\rM(\widehat{\cM}, \vartheta^A)\le \mdim_\rM(\widehat{(\Zb\Gamma)^n/\cM_j}, \vartheta^{A_j})$ for every $j\in \Nb$. It follows that $\mdim_\rM(\widehat{\cM}, \vartheta^A)\le \mdim(\widehat{\cM})$.
By Lemma~\ref{L-pseudometric to metric} we have
$\mdim(\widehat{\cM})\le \mdim_\rM(\widehat{\cM}, \vartheta^A)$. Therefore $\mdim(\widehat{\cM})=\mdim_\rM(\widehat{\cM}, \vartheta^A)$.
\end{proof}

Next we discuss metric mean dimension for inverse limits. It will enable us to pass from finitely generated modules to countable modules in the proof of Theorem~\ref{T-metric mdim}.
For a sequence of topological spaces $\{X_j\}_{j\in \Nb}$ with continuous maps $\pi_j: X_{j+1}\rightarrow X_j$ for all $j\in \Nb$, the inverse limit
$\varprojlim_{j\to \infty}X_j$ is defined as the subspace of $\prod_{j\in \Nb}X_j$ consisting of all elements $(x_j)_{j\in \Nb}$ satisfying $\pi_j(x_{j+1})=x_j$ for all $j\in \Nb$.

\begin{lemma} \label{L-metric mdim of projective limit}
Let $\{X_j\}_{j\in \Nb}$ be a sequence of compact metrizable spaces carrying continuous $\Gamma$-actions and continuous  $\Gamma$-equivariant maps $\pi_j: X_{j+1}\rightarrow X_j$ for all $j\in \Nb$. Let $\rho_j$ be a continuous pseudometric on $X_j$ for each $j\in \Nb$ such that $\rho_j(\pi_j(x), \pi_j(y))\le \rho_{j+1}(x, y)$ for all $x, y\in X_{j+1}$. Then there is a decreasing sequence $\{\lambda_j\}_{j\in \Nb}$ of positive numbers such that the continuous pseudometric
on $X:=\varprojlim_{j\to \infty}X_j$ defined by
$$\rho((x_j)_j, (y_j)_j)=\max_{j\in \Nb}\lambda_j\rho_j(x_j, y_j)$$
satisfies
$$\mdim_\rM(X, \rho)\le \varliminf_{j\to \infty}\mdim_\rM(X_j, \rho_j).$$
\end{lemma}
\begin{proof} We shall require $\lambda_j\diam(X_j, \rho_j)<1/2^j$ for all $j\in \Nb$, which implies that $\rho$ is a continuous pseudometric on $X$.
We shall also require $\lambda_{j+1}\le \lambda_j\le 1$ for all $j\in \Nb$.

Let $k\in \Nb$. Define a continuous pseudometric $\rho_k'$ on $X_k$ by $\rho_k'(x, y)=\max_{1\le j\le k}\lambda_j\rho_j(\pi_j\circ \pi_{j+1}\circ \cdots \circ\pi_{k-1}(x), \pi_j\circ \pi_{j+1}\circ \cdots \circ \pi_{k-1}(y))$.  Then $\rho_k'\le \rho_k$.
Thus
$$\mdim_\rM(X_k, \rho_k')\le \mdim_\rM(X_k, \rho_k).$$
Then we can find some $0<\varepsilon_k<\frac{1}{k}$ such that
$$ \frac{1}{|\log \varepsilon_k|}\varlimsup_F\frac{\log N_{\varepsilon_k}(X_k, (\rho_k')_F)}{|F|}\le \mdim_\rM(X_k, \rho_k')+\frac{1}{k}\le \mdim_\rM(X_k, \rho_k)+\frac{1}{k}.$$
Denote by $\Pi_k$ the natural map $X\rightarrow X_k$. Note that
$$\rho(x, y)\le \max(\rho_k'(\Pi_k(x), \Pi_k(y)), \max_{j>k}\lambda_j\diam(X_j, \rho_j))$$
for any $x, y\in X$. It follows that for any $F\in \cF(\Gamma)$, one has
$$\rho_F(x, y)\le \max((\rho_k')_F(\Pi_k(x), \Pi_k(y)), \max_{j>k} \lambda_j\diam(X_j, \rho_j))$$
for any $x, y\in X$.
Thus for any $\varepsilon>\max_{j>k} \lambda_j\diam(X_j, \rho_j)$, if $\cW\subseteq X$ is $(\rho_F, \varepsilon)$-separated, then $\Pi_k(\cW)$ is $((\rho_k')_F, \varepsilon)$-separated.
Therefore, if $\max_{j>k} \lambda_j\diam(X_j, \rho_j)<\varepsilon_k$, then
$$  N_{\varepsilon_k}(X, \rho_F)\le N_{\varepsilon_k}(X_k, (\rho_k')_F)$$
for every $F\in \cF(\Gamma)$, and hence
\begin{align*}
\frac{1}{|\log \varepsilon_k|}\varlimsup_F\frac{\log N_{\varepsilon_k}(X, \rho_F)}{|F|}&\le \frac{1}{|\log \varepsilon_k|}\varlimsup_F\frac{\log N_{\varepsilon_k}(X_k, (\rho_k')_F)}{|F|}\\
&\le \mdim_\rM(X_k, \rho_k)+\frac{1}{k}.
\end{align*}

Now we require further that $\max_{j>k} \lambda_j\diam(X_j, \rho_j)<\varepsilon_k$ for all $k\in \Nb$. Since we can choose $\varepsilon_k$ once $\lambda_1, \dots, \lambda_k$ are given, by induction such a sequence $\{\lambda_j\}_{j\in \Nb}$ exists. Then
\begin{align*}
\mdim_\rM(X, \rho)&\le \varliminf_{k\to \infty} \frac{1}{|\log \varepsilon_k|}\varlimsup_F\frac{\log N_{\varepsilon_k}(X, \rho_F)}{|F|}\\
&\le \varliminf_{k\to \infty}(\mdim_\rM(X_k, \rho_k)+\frac{1}{k}) \\
&= \varliminf_{k\to \infty}\mdim_\rM(X_k, \rho_k).
\end{align*}
\end{proof}

We are ready to prove Theorem~\ref{T-metric mdim}.

\begin{proof}[Proof of Theorem~\ref{T-metric mdim}] By Pontryagin duality we have $X=\widehat{\cM}$ for some countable $\Zb\Gamma$-module.
List the elements of $\cM$ as $a_1, a_2, \cdots$. Set $A_j=\{a_1, \dots, a_j\}$ and denote by $\cM_j$ the submodule of $\cM$ generated by $A_j$ for each $j\in \Nb$. By Lemma~\ref{L-finitely generated case metric mdim} one has
$\mdim(\widehat{\cM_j})=\mdim_\rM(\widehat{\cM_j}, \vartheta^{A_j})$ for each $j\in \Nb$.

For each $j\in \Nb$, from the inclusion $\cM_j\hookrightarrow \cM_{j+1}$ we have a surjective $\Gamma$-equivariant continuous map $\pi_j: \widehat{\cM_{j+1}}\rightarrow \widehat{\cM_j}$.
Then $\widehat{\cM}=\varprojlim_{j\to \infty}\widehat{\cM_j}$. Note that $\vartheta^{A_{j+1}}(x, y)\ge \vartheta^{A_j}(\pi_j(x), \pi_j(y))$ for all $j\in \Nb$ and $x, y\in \widehat{\cM_{j+1}}$. By Lemma~\ref{L-metric mdim of projective limit} we can find a suitable decreasing sequence $\{\lambda_j\}_{j\in \Nb}$ of positive numbers such that for the continuous pseudometric $\rho$  on $\widehat{\cM}=\varprojlim_{j\to \infty}\widehat{\cM_j}$ defined by
$$\rho((x_j)_j, (y_j)_j):=\max_{j\in \Nb}\lambda_j\vartheta^{A_j}(x_j, y_j),$$
one has $\mdim_\rM(\widehat{\cM}, \rho)\le \varliminf_{j\to \infty}\mdim_\rM(\widehat{\cM_j}, \vartheta^{A_j})\le \sup_{j\in \Nb}\mdim_\rM(\widehat{\cM_j}, \vartheta^{A_j})$.
Clearly
$$\rho(x, y)=\max_{j\in \Nb}\lambda_j\max_{1\le k\le j}\vartheta(x(a_k), y(a_k))=\max_{j\in \Nb}\lambda_j\vartheta(x(a_j), y(a_j))$$
for all $x, y\in \widehat{\cM}$. Thus $\rho$ is a compatible translation-invariant metric on $\widehat{\cM}$, and hence $\mdim(\widehat{\cM})\le \mdim_\rM(\widehat{\cM}, \rho)$.

Now we have \begin{align*}
\mdim(\widehat{\cM})\le \mdim_\rM(\widehat{\cM}, \rho)\le \sup_{j\in \Nb}\mdim_\rM(\widehat{\cM_j}, \vartheta^{A_j})=\sup_{j\in \Nb}\mdim(\widehat{\cM_j})\le\mdim(\widehat{\cM}),
\end{align*}
where the last inequality follows from Corollary~\ref{C-addition mdim}.
Therefore $\mdim(\widehat{\cM})=\mdim_\rM(\widehat{\cM}, \rho)$.
\end{proof}

\section{Range of mean dimension} \label{S-range}

In this section we use positive results on the strong Atiyah conjecture to discuss the range of mean dimension for algebraic actions.

Let $\Gamma$ be a discrete group.
Denote by $\cH(\Gamma)$ the subgroup of $\Qb$ generated by $|H|^{-1}$ for $H$ ranging over all finite subgroups of $\Gamma$. Motivated by a question of Atiyah on rationality of $L^2$-Betti numbers, one has the following {\it strong Atiyah conjecture} \cite[Conjecture 10.2]{Luck}:

\begin{conjecture}\label{C-Atiyah}
For any discrete (not necessarily amenable) group $\Gamma$ and any $f\in M_{m,n}(\Cb\Gamma)$ for some $m,n \in \Nb$,
denoting by $P_f$ the kernel of the orthogonal projection from $(\ell^2(\Gamma))^{n\times 1}$ to $\ker f$, one has
$$\tr_{\cN\Gamma}P_f\in \cH(\Gamma).$$
\end{conjecture}

A counterexample to Conjecture~\ref{C-Atiyah} has been found by Grigorchuk and \.{Z}uk \cite{GZ} (see also \cite{Austin, DS, Grabowski1, Grabowski2, GLSZ, LWa, PSZ}).
On the other hand, Conjecture~\ref{C-Atiyah} has been verified for various groups. In particular, Linnell has proved the following theorem \cite[Theorem 1.5]{Linnell} \cite[Theorem 10.19]{Luck}. Recall that the class of {\it elementary amenable groups} is the smallest class of groups containing all finite groups and abelian groups and being closed under taking subgroups, quotient groups, group extensions and directed unions \cite{Chou}.

\begin{theorem} \label{T-Linnell}
Let $\cC$ be the smallest class of groups containing all free groups and being closed under directed unions and extensions with elementary amenable quotients.
If a (not necessarily amenable) group $\Gamma$ belongs to $\cC$ and there is an upper bound on the orders of the finite subgroups of $\Gamma$, then Conjecture~\ref{C-Atiyah} holds for $\Gamma$.
\end{theorem}

The next result follows from Theorem~\ref{T-Linnell} and \cite[Lemma 10.10]{Luck}. For the convenience of the reader, we give a proof here.

\begin{theorem} \label{T-vrank rational}
Let $\Gamma$ be an elementary amenable group with an upper bound on the orders of the finite subgroups of $\Gamma$. For any (left) $\Zb\Gamma$-module $\cM$, one has $\vrank(\cM)\in \cH(\Gamma)\cup \{+\infty\}$.
\end{theorem}
\begin{proof} By Theorem~\ref{T-Linnell} we know that Conjecture~\ref{C-Atiyah} holds for $\Gamma$. By Lemma~\ref{L-ker and dim} we have $\vrank(\cM)\in \cH(\Gamma)$ for every finitely presented $\Zb\Gamma$-module $\cM$. Note that $\cH(\Gamma)$ is a discrete subset of $\Rb$. Then the argument in the proof of Theorem~\ref{T-mrank vs vrank} shows that $\vrank(\cM)\in \cH(\Gamma)$ for every finitely generated $\Zb\Gamma$-module $\cM$, and $\vrank(\cM)\in \cH(\Gamma)\cup \{+\infty\}$ for an arbitrary $\Zb\Gamma$-module $\cM$.
\end{proof}

Coornaert and Krieger showed that if a countable amenable group $\Gamma$ has subgroups of arbitrary large finite index, then for any $r\in [0, +\infty]$ there is a continuous action of $\Gamma$ on some compact metrizable space $X$ with $\mdim(X)=r$ \cite{CK}. Thus it is somehow surprising that
the value of the mean dimension of algebraic actions of some amenable groups is rather restricted, as the following consequence of Theorems~\ref{T-main} and \ref{T-vrank rational} shows: (see also \cite{CT} for some related discussion)

\begin{corollary} \label{C-mdim rational}
 Let $\Gamma$ be an elementary amenable group with an upper bound on the orders of the finite subgroups of $\Gamma$. For any (left) $\Zb\Gamma$-module $\cM$, one has $\mdim(\widehat{\cM}), \mrank(\cM)\in \cH(\Gamma)\cup \{+\infty\}$.
\end{corollary}

\section{Zero mean dimension} \label{S-zero}

Throughout this section $\Gamma$ will be a discrete countable amenable group.

For a continuous action of $\Gamma$ on a compact metrizable space $X$ and any finite open cover $\cU$ of $X$, the limit $\lim_F\frac{\log N(\cU^F)}{|F|}$ exists
by the Ornstein-Weiss lemma \cite{OW} \cite[Theorem 6.1]{LW}, where $N(\cU^F)$ denotes $\min_{\cV}|\cV|$ for $\cV$ ranging over  subcovers of $\cU^F$. The {\it topological entropy} of the action $\Gamma\curvearrowright X$, denoted by $\rh(X)$, is defined as $\sup_{\cU}\lim_F\frac{\log N(\cU^F)}{|F|}$ for $\cU$ ranging over finite open covers of $X$. Lindenstrauss and Weiss showed that if $\rh(X)<+\infty$, then $\mdim(X)=0$ \cite[Section 4]{LW}.

The following question was raised by Lindenstrauss implicitly in \cite{Lindenstrauss}:

\begin{question} \label{Q-zero}
Let $\Gamma$ act on a compact metrizable space $X$ continuously. Is it true that $\mdim(X)=0$ if and only if the action $\Gamma\curvearrowright X$ is the inverse limit of actions with finite topological entropy?
\end{question}

Lindenstrauss proved the ``if'' part \cite[Proposition 6.11]{Lindenstrauss}. For $\Gamma=\Zb$ Lindenstrauss showed that the ``only if'' part holds when the action $\Zb\curvearrowright X$ has an infinite minimal factor \cite[Proposition 6.14]{Lindenstrauss}, and  Gutman showed that the ``only if'' part holds when
$\Zb\curvearrowright X$ has the so-called {\it marker property} \cite[Definition 3.1, Theorem 8.3]{Gutman13}, in particular when $\Zb\curvearrowright X$ has an aperiodic factor $\Zb\curvearrowright Y$ such that either $Y$ is finite-dimensional or $\Zb\curvearrowright Y$ has only a countable number of closed invariant minimal subsets or $\Zb\curvearrowright Y$ has a so-called {\it compact minimal subsystems selector}
\cite[Theorem 8.4, Definition 3.8]{Gutman13}.

We shall answer Question~\ref{Q-zero} for algebraic actions of the groups in Theorem~\ref{T-vrank rational}. We start with finitely presented $\Zb\Gamma$-modules.

\begin{corollary} \label{C-entropy and mdim finitely presented}
Let $\cM$ be a finitely presented $\Zb\Gamma$-module. Then $\mdim(\widehat{\cM})=0$ if and only if $\rh(\widehat{\cM})<+\infty$.
\end{corollary}
\begin{proof} By Theorem~\ref{T-main} we know that $\mdim(\widehat{\cM})=0$ if and only if $\vrank(\cM)=0$,  while by  \cite[Remark 5.2]{LT} for any finitely presented $\Zb\Gamma$-module $\cM$ one has $\vrank(\cM)=0$ if and only if $\rh(\widehat{\cM})$ is finite.
\end{proof}

\begin{remark} \label{R-zero}
Corollary~\ref{C-entropy and mdim finitely presented} does not hold for arbitrary $\Zb\Gamma$-modules.
For example, for $\cM=\bigoplus_{n\ge 2}(\Zb/n\Zb)\Gamma$, one has $\mdim(\widehat{\cM})=\mrank(\cM)=0$ since $\cM$ is torsion as an abelian group,
while  $\rh(\widehat{\cM})=+\infty$ since $\rh(\widehat{\cM})\ge \rh(\widehat{(\Zb/n\Zb)\Gamma})=\log n$ for every $n\ge 2$. As another example, if we take $\Gamma=\Zb$ and $\cM$ to be the direct sum of infinite copies of $\Zb\Gamma/\Zb\Gamma f$ for $f=1+2T$, where we identify $\Zb\Gamma$ with $\Zb[T^\pm]$ naturally, then $\mdim(\widehat{\cM})=\mrank(\cM)=0$ since $\mrank(\Zb\Gamma/\Zb\Gamma f)=0$, while $\rh(\widehat{\cM})=+\infty$ since $\rh(\widehat{\Zb\Gamma/\Zb\Gamma f})=\log 2>0$ \cite[Propositions 16.1 and 17.2]{Schmidt}. In the second example, $\cM$ is torsion-free as an abelian  group.
\end{remark}

None of  the above examples is finitely generated. This leads us to the following question:

\begin{question} \label{Q-finitely generated zero}
For any finitely generated $\Zb\Gamma$-module $\cM$, if $\mdim(\widehat{\cM})=0$, then must $\rh(\widehat{\cM})$ be finite?
\end{question}

We answer Question~\ref{Q-finitely generated zero} affirmatively for the groups in Theorem~\ref{T-vrank rational}:

\begin{corollary} \label{C-entropy and mdim finitely generated}
Let $\Gamma$ be an elementary amenable group with an upper bound on the orders of the finite subgroups of $\Gamma$.
Let $\cM$ be a finitely generated $\Zb\Gamma$-module. Then $\mdim(\widehat{\cM})=0$ if and only if $\rh(\widehat{\cM})<+\infty$.
\end{corollary}
\begin{proof} We just need to show the ``only if'' part.
We have $\cM=(\Zb\Gamma)^m/\cM_\infty$ for some $m\in \Nb$ and some submodule $\cM_\infty$ of $(\Zb\Gamma)^m$.
Write $\cM_\infty$ as the union of an increasing net of finitely generated submodules $\{\cM_n\}_{n\in J}$ with $\infty\not \in J$.  Then
\begin{align*}
0=\mdim(\widehat{\cM})\overset{Theorem~\ref{T-main}}=\vrank(\cM)\overset{Lemma~\ref{L-decreasing limit vrank}}=\lim_{n\to \infty} \vrank((\Zb\Gamma)^m/\cM_n).
\end{align*}
By Theorem~\ref{T-vrank rational} we know that $\vrank((\Zb\Gamma)^m/\cM_n)$ is contained in the discrete subset $\cH(\Gamma)$ of $\Rb$ for every $n\in J$.
Thus, when $n$ is sufficiently large, one has $\vrank((\Zb\Gamma)^m/\cM_n)=0$.
Since $(\Zb\Gamma)^m/\cM_n$ is a finitely presented $\Zb\Gamma$-module, by Corollary~\ref{C-entropy and mdim finitely presented} we have $\rh(\widehat{(\Zb\Gamma)^m/\cM_n})<+\infty$. Because $\widehat{\cM}$ is a closed $\Gamma$-invariant subset of $\widehat{(\Zb\Gamma)^m/\cM_n}$, we get
$\rh(\widehat{\cM})\le \rh(\widehat{(\Zb\Gamma)^m/\cM_n})<+\infty$.
\end{proof}

Since every countable $\Zb\Gamma$-module is the union of an increasing sequence of finitely generated submodules, and by Pontryagin duality every $\Gamma$-action on a compact metrizable abelian group by continuous automorphisms arises from some countable $\Zb\Gamma$-module, from Corollaries~\ref{C-addition mdim} and \ref{C-entropy and mdim finitely generated} we answer Question~\ref{Q-zero} affirmatively for algebraic actions of groups in Theorem~\ref{T-vrank rational}:

\begin{corollary} \label{C-zero}
Let $\Gamma$ be an elementary amenable group with an upper bound on the orders of the finite subgroups of $\Gamma$, and let $\Gamma$ act on a compact metrizable abelian group $X$ by continuous automorphisms. Then $\mdim(X)=0$ if and only if the action $\Gamma\curvearrowright X$ is the inverse limit of actions with finite topological entropy.
\end{corollary}


\appendix
\section{Mean rank and Elek rank} \label{S-Elek}

In this appendix we discuss the relation between the mean rank of a $\Zb\Gamma$-module and the Elek rank of a finitely generated $\Qb\Gamma$-module.

Let $R$ be an integral domain and $\Gamma$ be a discrete amenable group. As we explained in Remark~\ref{R-SVV}, one may extend the definition of mean rank for $\Zb\Gamma$-modules in Section~\ref{S-mean rank} to $R\Gamma$-modules naturally as follows. Denote by $K$ the field of fractions of $R$. For any $R$-module $\sM$, the rank of $\sM$ is defined as $\dim_K(K\otimes_R\sM)$ and denoted by $\rank_R(\sM)$. For any $R$-module $\sM$ and $A\subseteq \sM$, denote by $\left<A\right>_R$ the submodule of $\sM$ generated by $A$. Then we define the {\it mean rank} of a (left) $R\Gamma$-module $\cM$ as
$$ \mrank_R(\cM):=\sup_{A\in \cF(\cM)}\lim_F \frac{\rank_R(\left<F^{-1}A\right>_R)}{|F|}=\sup_{A\in \cF(\cM)}\inf_{F\in \cF(\Gamma)} \frac{\rank_R(\left<F^{-1}A\right>_R)}{|F|}.$$
All the results in Section~\ref{S-mean rank} hold for $R\Gamma$-modules without change of proof.

\begin{proposition} \label{P-fraction field}
For any (left) $R\Gamma$-module $\cM$, one has
$$\mrank_R(\cM)=\mrank_K(K\Gamma\otimes_{R\Gamma}\cM). $$
\end{proposition}
\begin{proof} Denote by $\varphi$ the natural $R\Gamma$-module homomorphism $\cM\rightarrow K\Gamma\otimes_{R\Gamma}\cM$ sending $a$ to $1\otimes a$.
Note that
$$ K\Gamma\otimes_{R\Gamma}\cM=(K\otimes_RR\Gamma)\otimes_{R\Gamma}\cM=K\otimes_R(R\Gamma\otimes_{R\Gamma}\cM)=K\otimes_R\cM.$$
It follows that
$$ \mrank_K(K\Gamma\otimes_{R\Gamma}\cM)=\sup_{A\in \cF(\cM)}\lim_F \frac{\rank_K(\left<F^{-1}\varphi(A)\right>_K)}{|F|}.$$
Thus it suffices to show that for any $A\in \cF(\cM)$ and $F\in \cF(\Gamma)$ one has
$$\rank_R(\left<F^{-1}A\right>_R)=\rank_K(\left<F^{-1}\varphi(A)\right>_K).$$
Since the functor $K\otimes_R\cdot$ is exact \cite[Proposition XVI.3.2]{Lang}, from the inclusion $\left<F^{-1}A\right>_R\hookrightarrow \cM$, we have the inclusion
$K\otimes_R(\left<F^{-1}A\right>_R)\hookrightarrow K\otimes_R\cM=K\Gamma\otimes_{R\Gamma}\cM$, with image $\left<F^{-1}\varphi(A)\right>_K$. Therefore
$$ \rank_K(\left<F^{-1}\varphi(A)\right>_K)=\rank_K(K\otimes_R(\left<F^{-1}A\right>_R))=\rank_R(\left<F^{-1}A\right>_R).$$
\end{proof}

For any field $K$ and any discrete amenable group $\Gamma$, in \cite{Elek1} Elek introduced a rank function for finitely generated (left) $K\Gamma$-modules $\cM$ as follows. Write $\cM$ as $(K\Gamma)^{n\times 1}/\cM'$ for some $n\in \Nb$ and some submodule $\cM'$ of $(K\Gamma)^{n\times 1}$. For any vector space $V$ over $K$, denote by $V^*$ the dual vector space of $V$, consisting of all $K$-linear maps $V\rightarrow K$. One may identify $((K\Gamma)^{n\times 1})^*$ with $(K^{1\times n})^\Gamma$ via
$$ f(a)=\sum_{s\in \Gamma}f_s\cdot a_s$$
for all $f\in (K^{1\times n})^\Gamma$  and $a\in (K\Gamma)^{n\times 1}$.
Then one can identify $\cM^*$ with $\{f\in (K^{1\times n})^\Gamma: f(\cM')=0\}$.
For each $F\in \cF(\Gamma)\cup \{\emptyset\}$, denote by $\cM^*|_F$ the image of $\cM^*$ under the restriction map $(K^{1\times n})^\Gamma\rightarrow (K^{1\times n})^F$.
The function $\cF(\Gamma)\cup \{\emptyset\}\rightarrow \Zb$ sending $F$ to $\dim_K(\cM^*|_{F^{-1}})$  is easily seen to satisfy the conditions of the Ornstein-Weiss lemma \cite{OW} \cite[Theorem 6.1]{LW} and hence the limit $\lim_F\frac{\dim_K(\cM^*|_{F^{-1}})}{|F|}$ exists. (Actually one can check that this function $\cF(\Gamma)\cup \{\emptyset\}\rightarrow \Zb$ satisfies the conditions in Lemma~\ref{L-property of rank}, and hence $\lim_F\frac{\dim_K(\cM^*|_{F^{-1}})}{|F|}=\inf_{F\in \cF(\Gamma)}\frac{\dim_K(\cM^*|_{F^{-1}})}{|F|}$, though we shall not need this fact.)
The {\it Elek rank} of $\cM$ is defined to be this limit, and denoted by $\erank_K(\cM)$. The Elek rank satisfies the addition formula \cite{Elek1}.

Note that $\mrank_K(\cM)$ is defined using the module $\cM$ directly, which is crucial for the proof of Theorem~\ref{T-main} since to compute the mean dimension or the von Neumann-L\"{u}ck rank one has to pass to the Pontryagin dual $\widehat{\cM}$ or the tensor product $\cN\Gamma\otimes_{\Zb\Gamma} \cM$ first, while $\erank_K(\cM)$ is defined by passing to the dual $\cM^*$ first. So a priori there is no relation between these two ranks. The fact that these two ranks share some properties such as the addition formula hints that they might be the same.  We show that indeed this is the case:

\begin{theorem} \label{T-Elek rank}
For any finitely generated $K\Gamma$-module $\cM$, one has
$$\mrank_K(\cM)=\erank_K(\cM).$$
\end{theorem}
\begin{proof} Denote by $A$ the image of the standard $K\Gamma$-basis of $(K\Gamma)^{n\times 1}$ under the quotient map $(K\Gamma)^{n\times 1}\rightarrow \cM$.
Then $A$ is a finite generating subset of $\cM$. By Lemma~\ref{L-finitely generated} it suffices to show
$$\lim_F\frac{\dim_K(\left<F^{-1}A\right>_K)}{|F|}=\lim_F\frac{\dim_K(\cM^*|_{F^{-1}})}{|F|}.$$
Therefore it suffices to show $\dim_K(\left<F^{-1}A\right>_K)=\dim_K(\cM^*|_{F^{-1}})$ for every $F\in \cF(\Gamma)$.
Denote  by $T$ the linear map
$\cM^*\rightarrow K^{F^{-1}A}$ sending $f$ to $b\mapsto f(b)$ for all $b\in F^{-1}A$.
 Since every $K$-linear map $\left<F^{-1}A\right>_K\rightarrow K$ extends to a $K$-linear map $\cM\rightarrow K$, one has
 $$\dim_K(\left<F^{-1}A\right>_K)=\dim_K((\left<F^{-1}A\right>_K)^*)=\dim_K(T(\cM^*)).$$
One may identify $T(\cM^*)$ with $\cM^*|_{F^{-1}}$ naturally. Therefore
$$  \dim_K(\left<F^{-1}A\right>_K)=\dim_K(T(\cM^*))=\dim_K(\cM^*|_{F^{-1}}).$$
\end{proof}

From Proposition~\ref{P-fraction field} and Theorem~\ref{T-Elek rank} we obtain

\begin{corollary} \label{C-Elek rank}
For any finitely generated $\Zb\Gamma$-module $\cM$, one has
$$\mrank(\cM)=\erank_\Qb(\Qb\Gamma\otimes_{\Zb\Gamma}\cM).$$
\end{corollary}



\begin{thebibliography}{99}
\Small

\bibitem{AF}
F.~W.~Anderson and K.~R.~Fuller. {\it Rings and Categories of Modules}. Second edition. Graduate Texts in Mathematics, 13. Springer-Verlag, New York, 1992.

\bibitem{Atiyah}
M.~F.~Atiyah. Elliptic operators, discrete groups and von Neumann algebras. In: {\it Colloque "Analyse et Topologie'' en l'Honneur de Henri Cartan (Orsay, 1974)}, pp. 43--72. Asterisque, no. 32--33, Soc. Math. France, Paris, 1976.

\bibitem{Austin}
T.~Austin. Rational group ring elements with kernels having irrational dimension.  	{\it Proc. London Math. Soc.} {\bf 107} (2013), no. 6, 1424--1448.

\bibitem{CC}
T.~Ceccherini-Silberstein and M.~Coornaert. {\it Cellular Automata and Groups}. Springer Monographs in Mathematics. Springer-Verlag, Berlin, 2010.

\bibitem{Chou}
C.~Chou. Elementary amenable groups. {\it Illinois J. Math.} {\bf 24} (1980), no. 3, 396--407.

\bibitem{CL}
N.-P.~Chung and H.~Li. Homoclinic groups, IE groups, and expansive algebraic actions. {\it Invent. Math.} {\bf 199} (2015), no. 3, 805--858.

\bibitem{CT}
N.-P.~Chung and A.~Thom. Some remarks on the entropy for algebraic actions of amenable groups.  {\it Trans. Amer. Math. Soc.} to appear.

\bibitem{CK}
M.~Coornaert and F.~Krieger. Mean topological dimension for actions of discrete amenable groups. {\it Discrete Contin. Dyn. Syst.} {\bf 13} (2005), no. 3, 779--793.

\bibitem{DS}
W.~Dicks and T.~Schick. The spectral measure of certain elements of the complex group ring of a wreath product. {\it Geom. Dedicata} {\bf 93} (2002), 121--137.

\bibitem{Elek}
G.~Elek. On the analytic zero divisor conjecture of Linnell.
{\it Bull. London Math. Soc.}  {\bf 35}  (2003),  no. 2, 236--238.

\bibitem{Elek1}
G.~Elek. The rank of finitely generated modules over group algebras. {\it Proc. Amer. Math. Soc.} {\bf 131} (2003), no. 11, 3477--3485.

\bibitem{Grabowski2}
{\L}.~Grabowski. Irrational $\ell^2$-invariants arsing from the lamplighter group. arXiv:1009.0229.


\bibitem{Grabowski1}
{\L}.~Grabowski. On Turing dynamical systems and the Atiyah problem.  {\it Invent. Math.} {\bf 198} (2014), no. 1, 27--69.


\bibitem{GZ}
R.~I.~Grigorchuk abd A.~\.{Z}uk. The lamplighter group as a group generated by a $2$-state automaton, and its spectrum. {\it Geom. Dedicata} {\bf 87} (2001), no. 1-3, 209--244.

\bibitem{GLSZ}
R.~I.~Grigorchuk, P.~Linnell, T.~Schick, and A.~\.{Z}uk. On a question of Atiyah. {\it C. R. Acad. Sci. Paris S\'{e}r. I Math.} {\bf 331} (2000), no. 9, 663--668.

\bibitem{Gromov}
M.~Gromov. Topological invariants of dynamical systems and spaces of holomorphic maps. I.
{\it Math. Phys. Anal. Geom.} {\bf 2} (1999), no. 4, 323--415.

\bibitem{Gutman}
Y.~Gutman. Embedding $\Zb^k$-actions in cubical shifts and $\Zb^k$-symbolic extensions. {\it Ergod.\ Th.\ Dynam.\ Sys.} {\bf 31} (2011), no. 2, 383--403.

\bibitem{Gutman13}
Y.~Gutman. Dynamical embedding in cubical shifts \& the topological Rokhlin and small boundary properties. arXiv:1301.6072.

\bibitem{HW}
W.~Hurewicz and H.~Wallman. {\it Dimension Theory}.
Princeton Mathematical Series, v. 4. Princeton University Press, Princeton, N. J., 1941.

\bibitem{Karube}
T.~Karube. On the local cross-sections in locally compact groups. {\it J. Math. Soc. Japan} {\bf 10} (1958), 343--347.

\bibitem{KR2}
R.~V.~Kadison and J.~R.~Ringrose. {\it Fundamentals of the Theory of Operator Algebras. Vol. II. Advanced Theory.}
Graduate Studies in Mathematics, 16. American Mathematical Society, Providence, RI, 1997.

\bibitem{Krieger06}
F.~Krieger. Groupes moyennables, dimension topologique moyenne et sous-d\'{e}calages. (French) [Amenable groups, mean topological dimension and subshifts]
{\it Geom. Dedicata} {\bf 122} (2006), 15--31.

\bibitem{Krieger09}
F.~Krieger. Minimal systems of arbitrary large mean topological dimension. {\it Israel J. Math.} {\bf 172} (2009), 425--444.

\bibitem{Lang}
S.~Lang. {\it Algebra}. Revised third edition. Graduate Texts in Mathematics, 211. Springer-Verlag, New York, 2002.

\bibitem{LWa}
F.~Lehner and S.~Wagner. Free Lamplighter groups and a question of Atiyah. {\it Amer. J. Math.} {\bf 135} (2013), no. 3, 835--849.


\bibitem{Li}
H.~Li. Compact group automorphisms, addition formulas and Fuglede-Kadison determinants. {\it Ann. of Math. (2)} {\bf 176} (2012), no. 1, 303--347.

\bibitem{Li13}
H.~Li. Sofic mean dimension. {\it Adv. Math.} {\bf 244} (2013), 570--604.


\bibitem{LT}
H.~Li and A.~Thom. Entropy, determinants, and $L^2$-torsion. {\it J. Amer. Math. Soc.} {\bf 27} (2014), no. 1, 239--292.

\bibitem{LSW}
D.~Lind, K.~Schmidt, and T.~Ward. Mahler measure and
entropy for commuting automorphisms of compact groups.
{\it Invent.\ Math.}  {\bf 101}  (1990), 593--629.

\bibitem{Lindenstrauss}
E.~Lindenstrauss. Mean dimension, small entropy factors and an embedding theorem.
{\it Inst. Hautes \'{E}tudes Sci. Publ. Math.} {\bf 89} (1999), 227--262.


\bibitem{LW}
E.~Lindenstrauss and B.~Weiss. Mean topological dimension.
{\it Israel J.\ Math.} {\bf 115} (2000), 1--24.

\bibitem{Linnell}
P.~A.~Linnell. Division rings and group von Neumann algebras. {\it Forum Math.} {\bf 5} (1993), no. 6, 561--576.

\bibitem{Luck98}
W.~L\"{u}ck. Dimension theory of arbitrary modules over finite von Neumann algebras and $L^2$-Betti numbers. I. Foundations. {\it J. Reine Angew. Math.} {\bf 495} (1998), 135--162.

\bibitem{Luck}
W.~L\"{u}ck. {\it $L^2$-Invariants: Theory and Applications to Geometry and $K$-theory}.
Springer-Verlag, Berlin, 2002.

\bibitem{JMO}
J.~Moulin Ollagnier. {\it Ergodic Theory and Statistical Mechanics.}
Lecture Notes in Math., 1115. Springer, Berlin, 1985.

\bibitem{Nagata}
J.~Nagata. {\it Modern Dimension Theory}. Revised edition. Sigma Series in Pure Mathematics, 2. Heldermann Verlag, Berlin, 1983.

\bibitem{NR}
D.~G.~Northcott and M.~Reufel. A generalization of the concept of length. {\it Quart. J. Math. Oxford Ser. (2)} {\bf 16} (1965), 297--321.

\bibitem{OW}
D.~S.~Ornstein and B.~Weiss. Entropy and isomorphism theorems for
actions of amenable groups.  {\it J. Analyse Math.} {\bf 48} (1987),
1--141.

\bibitem{Peters}
J.~Peters. Entropy on discrete abelian groups.
{\it Adv. Math.} {\bf 33} (1979), no. 1, 1--13.

\bibitem{PSZ}
M.~Pichot, T.~Schick, and A.~\.{Z}uk. Closed manifolds with transcendental $L^2$-Betti numbers. arXiv:1005.1147.

\bibitem{Pisier}
G.~Pisier. {\it The Volume of Convex Bodies and Banach Space Geometry}. Cambridge Tracts in Mathematics, 94. Cambridge University Press, Cambridge, 1989.

\bibitem{Pontryagin}
L.~S.~Pontryagin. {\it Topological Groups}. Translated from the second Russian edition by Arlen Brown Gordon and Breach Science Publishers, Inc., New York-London-Paris, 1966.

\bibitem{PS}
L.~Pontrjagin and L.~Schnirelmann. Sur une propri\'{e}t\'{e} m\'{e}trique de la dimension.  {\it Ann. of Math. (2)} {\bf 33} (1932), no. 1, 156--162.
Translation in {\it Classics on Fractals}, pp. 133--142. Edited by G. A. Edgar. Studies in Nonlinearity. Westview Press. Advanced Book Program, Boulder, CO, 2004.

\bibitem{SVV}
L.~Salce, P.~V\'{a}mos, and S.~Virili. Length functions, multiplicities and algebraic entropy. {\it Forum Math.} {\bf 25} (2013), no. 2, 255--282.

\bibitem{Schmidt}
K.~Schmidt. {\it Dynamical Systems of Algebraic Origin}.
Progress in Mathematics, 128. Birkh\"{a}user Verlag, Basel, 1995.


\bibitem{Takesaki}
M.~Takesaki. {\it Theory of Operator Algebras. I.} Encyclopaedia of Mathematical Sciences, 124. Operator Algebras and Non-commutative Geometry, 5. Springer-Verlag, Berlin, 2002.

\bibitem{Tsukamoto08}
M.~Tsukamoto. Moduli space of Brody curves, energy and mean dimension. {\it Nagoya Math. J.} {\bf 192} (2008), 27--58.

\bibitem{Tsukamoto09}
M.~Tsukamoto. Deformation of Brody curves and mean dimension. {\it Ergod.\ Th.\ Dynam.\ Sys.} {\bf 29} (2009), no. 5, 1641--1657.

\bibitem{Yamanoshita}
T.~Yamanoshita. On the dimension of homogeneous spaces. {\it J. Math. Soc. Japan} {\bf 6} (1954), 151--159.

\end{thebibliography}
\end{document}